\documentclass[11pt]{amsart}

\usepackage[x11names, rgb]{xcolor}
\usepackage[utf8]{inputenc}
\usepackage{tikz}
\usetikzlibrary{snakes,arrows,shapes}

\usepackage{verbatim}
\usepackage{amssymb,amsthm,amsmath}
\usepackage{color}
\definecolor{darkblue}{rgb}{0, 0, .4}
\definecolor{grey}{rgb}{.7, .7, .7}
\usepackage[breaklinks]{hyperref}
\hypersetup{
	colorlinks=true,
	linkcolor=darkblue,
	anchorcolor=darkblue,
	citecolor=darkblue,
	pagecolor=darkblue,
	urlcolor=darkblue,
	pdftitle={},
	pdfauthor={}
}
	
\newtheorem{theorem}{Theorem}[section]
\newtheorem{lemma}[theorem]{Lemma}

\theoremstyle{definition}
\newtheorem{definition}[theorem]{Definition}
\newtheorem{example}[theorem]{Example}
\newtheorem{computerExample}[theorem]{Sage Example}

\theoremstyle{remark}
\newtheorem{remark}[theorem]{Remark}

\numberwithin{equation}{section}

\theoremstyle{theorem}
\newtheorem{corollary}[theorem]{Corollary}

\newtheorem{proposition}[theorem]{Proposition}
\newtheorem{conjecture}[theorem]{Conjecture}

\newcommand{\union}{\cup}

\newcommand{\gn}{ \bullet }  

\input xy
\xyoption{all}

\newcommand{\p}{\dot{p}}

\newcommand{\sage}{{\tt Sage} }

\newcommand{\sgn}{\mathrm{sgn}}
\newcommand{\wt}{\mathrm{wt}}
\newcommand{\zero}{\mathbf{0}}
\newcommand{\inner}[2]{\langle #1, #2 \rangle}

\begin{document}

\title{Affine structures and a tableau model for $\mathbf{E_6}$ crystals}

\begin{abstract}
We provide the unique affine crystal structure for type $E_6^{(1)}$
Kirillov--Reshetikhin crystals corresponding to the multiples of fundamental weights
$s\Lambda_1, s\Lambda_2$, and $s\Lambda_6$ for all $s\ge 1$ (in Bourbaki's
labeling of the Dynkin nodes, where $2$ is the adjoint node). Our
methods introduce a generalized tableaux model for classical highest weight
crystals of type $E$ and use the order three automorphism of the affine
$E_6^{(1)}$ Dynkin diagram. In addition, we provide a conjecture for the affine
crystal structure of type $E_7^{(1)}$ Kirillov--Reshetikhin crystals
corresponding to the adjoint node.
\end{abstract}

\author{Brant Jones}
\author{Anne Schilling}

\address{Department of Mathematics, One Shields Avenue, University of California, Davis, CA 95616}
\email{\href{mailto:[brant,anne]@math.ucdavis.edu}{\texttt{[brant,anne]@math.ucdavis.edu}}}
\urladdr{\url{http://www.math.ucdavis.edu/\~[brant,anne]/}}

\thanks{BJ was partially supported by NSF grant DMS-0636297.
AS was partially supported by the NSF grants DMS--0501101, DMS--0652641, and DMS--0652652.}


\date{\today}

\maketitle


\bigskip
\section{Introduction}\label{s:intro}

A uniform description of perfect crystals of level 1 corresponding to the highest 
root $\theta$ was given in~\cite{BFKL:2006}.
A generalization to higher level $s$ for certain nonexceptional types was studied  
in~\cite{kodera}. These crystals $B$ of level $s$ have the following decomposition
when removing the zero arrows \cite{Chari:2001}:
\begin{equation} \label{eq:adjoint decomp}
	B \cong \bigoplus_{k=0}^s B(k \theta),
\end{equation}
where $B(\lambda)$ denotes the highest weight crystal with highest weight
$\lambda$.

In this paper, we provide the unique affine crystal structure for the 
Kirillov--Reshetikhin crystals $B^{r,s}$ of type $E_6^{(1)}$ for the Dynkin nodes $r=1,2,$ and $6$ in 
the Bourbaki labeling, where node 2 corresponds to the adjoint node (see
Figure~\ref{fig:E67}).
In addition, we provide a conjecture for the affine crystal structure for type $E_7^{(1)}$
Kirillov--Reshetikhin crystals of level $s$ corresponding to the adjoint node.

Our construction of the affine crystals uses the classical decomposition~\eqref{eq:adjoint decomp}
 together with a promotion operator which yields the affine crystal operators.
Combinatorial models of all Kirillov--Reshetikhin crystals of nonexceptional types were
constructed using promotion and similarity methods in~\cite{schilling:2008,OS:2008,fos:2009a}.
Perfectness was proven in~\cite{fos:2009b}.
Affine crystals of type $E_6^{(1)}$ and $E_7^{(1)}$ of level 1 corresponding to minuscule 
coweights ($r=1,6$) were studied by Magyar~\cite{magyar:2006} using the Littelmann path model.
Hernandez and Nakajima~\cite{hernandez_nakajima:2006} gave a construction of the 
Kirillov--Reshetihkin crystals $B^{r,1}$ for all $r$ for type $E_6^{(1)}$ and most nodes $r$
in type $E_7^{(1)}$.

For nonexceptional types, the classical crystals appearing in the 
decomposition~\eqref{eq:adjoint decomp} can be described using Kashiwara--Nakashima
tableaux~\cite{KN}. We provide a similar construction for general types (see Theorem~\ref{t:tableaux}). 
This involves the explicit construction of the highest weight crystals $B(\Lambda_i)$ corresponding 
to fundamental weights $\Lambda_i$ using the Lenart--Postnikov~\cite{lenart-postnikov} model
and the notion of pairwise weakly increasing columns (see Definition~\ref{d:pwi}).

The promotion operator for the Kirillov--Reshetikhin crystal $B^{r,s}$ of type $E_6^{(1)}$ 
for $r=1,6$ is given in Theorem~\ref{t:E6_16} and for $r=2$ in Theorem~\ref{t:L2_main}.
Our construction and proofs exploit the notion of composition graphs 
(Definition~\ref{d:composition_graph}) and the fact that the promotion operator we choose has order 
three. As shown in Theorem~\ref{t:main}, a promotion operator of order three yields a regular crystal.
In Conjecture~\ref{conj:E7} we also provide a promotion operator of order two for the crystals 
$B^{1,s}$ of type $E_7^{(1)}$. However, for order two promotion operators 
the analogue of Theorem~\ref{t:main} is missing.

This paper is structured as follows. In Section~\ref{s:tableaux}, the fundamental crystals
$B(\Lambda_1)$ and $B(\Lambda_6)$ are constructed explicitly for type $E_6$ and it is 
shown that all other highest weight crystals $B(\lambda)$ of type $E_6$ can be constructed 
from these. Similarly, $B(\Lambda_7)$ yields all highest weight crystals $B(\lambda)$ for type
$E_7$. In Section~\ref{ss:gen_tab}, a generalized tableaux model is given for $B(\lambda)$
for general types. In particular, we introduce the notion of weak increase. The results are used 
to construct the affine crystals in Section~\ref{s:affine_structure}. In Section~\ref{s:sage_implementation},
we give some details about the \sage implementation of the $E_6$, $E_7$, and $E_6^{(1)}$ crystals
constructed in this paper. Some outlook and open problems are discussed in Section~\ref{s:outlook}.
Appendices~\ref{s:appendix} and~\ref{s:appendix_b} contain details
about the proofs for the construction of the affine crystals, in particular the usage of oriented 
matroid theory.

\subsection*{Acknowledgments}
We thank Daniel Bump for his interest in this work, reviewing some of our \sage code related
to $E_6$ and $E_7$, and his insight into connections of $B(\Lambda_1)$ of type $E_6$ and the Weyl
group action on 27 lines on a cubic surface.
We are grateful to Jesus DeLoera and Matthias Koeppe for their insights on oriented matroids.
We thank Masato Okado for pointing~\cite[Theorem 6.1]{KMOY:2007} out to us, and his comments 
and insights on earlier drafts of this work. We thank Satoshi Naito and Mark Shimozono 
for drawing our attention to monomial theory and references~\cite{LS:1986,littelmann:1996}.

For our computer explorations we used and implemented new features in the open-source
mathematical software \texttt{Sage}~\cite{sage} and its algebraic
combinatorics features developed by the \texttt{Sage-Combinat}
community~\cite{sage-combinat}; we are grateful to Nicolas M. Thi\'ery for all
his support.  Figure~\ref{fig:BLa7} was produced using \texttt{graphviz},
\texttt{dot2tex}, and \texttt{pgf/tikz}.


\bigskip
\section{A tableau model for finite-dimensional highest weight crystals}\label{s:tableaux}

In this section, we describe a model for the classical highest weight crystals
in type $E$.  In Section~\ref{ss:axiom}, we introduce our notation and give the
axiomatic definition of a crystal.  The tensor product rule for crystals is reviewed in
Section~\ref{ss:tensor}.  In Section~\ref{ss:fund_cr}, we give an
explicit construction of the highest weight crystals associated to the
fundamental weights in types $E_6$ and $E_7$.  In Section~\ref{ss:gen_tab}, we
give a generalized tableaux model to realize all of the highest weight crystals
in these types.  The generalized tableaux are type-independent, and can be
viewed as an extension of the Kashiwara--Nakashima tableaux \cite{KN} to type
$E$. For a general introduction to crystals we refer to~\cite{hk:2002}.

\subsection{Axiomatic definition of crystals} \label{ss:axiom}
Denote by $\mathfrak{g}$ a Lie algebra or symmetrizable Kac-Moody algebra, $P$ the weight
lattice, $I$ the index set for the vertices of the Dynkin diagram
of $\mathfrak{g}$, $\{\alpha_i\in P \mid i\in I \}$ the simple roots, and
$\{\alpha_i^\vee \in P^* \mid i\in I \}$ the simple coroots.
Let $U_q(\mathfrak{g})$ be the quantized universal enveloping algebra of
$\mathfrak{g}$. A \em $U_q(\mathfrak{g})$-crystal \em \cite{K:1995} is a nonempty set $B$
equipped with maps $\wt:B\rightarrow P$ and
$e_i,f_i:B\rightarrow B\cup\{\zero\}$ for all $i\in I$,
satisfying
\begin{align*}
f_i(b)=b' &\Leftrightarrow e_i(b')=b
\text{ if $b,b'\in B$} \\
\wt(f_i(b))&=\wt(b)-\alpha_i \text{ if $f_i(b)\in B$} \\
\label{eq:string length}
\inner{\alpha^\vee_i}{\wt(b)}&=\varphi_i(b)-\varepsilon_i(b).
\end{align*}
Here, we have 
\begin{equation*}
\begin{split}
\varepsilon_i(b)&= \max\{n\ge0\mid e_i^n(b)\not=\zero \} \\
\varphi_i(b) &= \max\{n\ge0\mid f_i^n(b)\not=\zero \}
\end{split}
\end{equation*}
for $b \in B$, and we denote $\inner{\alpha^\vee_i}{\wt(b)}$ by $\wt_i(b)$.
A $U_q(\mathfrak{g})$-crystal $B$ can be viewed as a directed edge-colored
graph called the \em crystal graph \em whose vertices are the elements of $B$,
with a directed edge from $b$ to $b'$ labeled $i\in I$, if and only if
$f_i(b)=b'$.  Given $i \in I$ and $b \in B$, the \em $i$-string through $b$ \em
consists of the nodes $\{ f_i^m(b) : 0 \leq m \leq \varphi_i(b) \} \cup \{
e_i^m(b) : 0 < m \leq \varepsilon_i(b) \}$.

Let $\{\Lambda_i \mid i\in I\}$ be the fundamental weights of $\mathfrak{g}$.
For every $b\in B$ define $\varphi(b)=\sum_{i\in I} \varphi_i(b) \Lambda_i$ and
$\varepsilon(b)=\sum_{i\in I} \varepsilon_i(b) \Lambda_i$. An element $b\in B$
is called \em highest weight \em if $e_i(b) = \zero$ for all $i\in I$.  We
say that $B$ is a \em highest weight crystal \em of highest weight $\lambda$ if
it has a unique highest weight element of weight $\lambda$.  For a dominant
weight $\lambda$, we let $B(\lambda)$ denote the unique highest-weight crystal
with highest weight $\lambda$.  

It follows from the general theory that every integrable
$U_q(\mathfrak{g})$-module decomposes as a direct sum of highest weight modules.
On the level of crystals, this implies that every crystal graph $B$
corresponding to an integrable module is a union of connected components, and
each connected component is the crystal graph of a highest weight module.  We
denote this by $B = \bigoplus B(\lambda)$ for some set of dominant weights
$\lambda$, and we call these $B(\lambda)$ the \em components \em of the crystal.

An \em isomorphism \em of crystals is a bijection $\Psi: B \cup \{\zero\}
\rightarrow B' \cup \{\zero\}$ such that $\Psi(\zero) = \zero$, $\varepsilon(\Psi(b)) =
\varepsilon(b)$, $\varphi(\Psi(b)) = \varphi(b)$, $f_i \Psi(b) = \Psi(f_i(b))$,
and $\Psi(e_i(c)) = e_i \Psi(c)$ for all $b,c \in B$, $\Psi(b),\Psi(c) \in B'$
where $f_i(b) = c$.

When $\widetilde{\lambda}$ is a weight in an affine type, we call 
\begin{equation}\label{e:level}
\langle \widetilde{\lambda}, c \rangle = \sum_{i \in I \cup \{0\}} a_i^{\vee} \langle \widetilde{\lambda}, \alpha_i^{\vee} \rangle
\end{equation}
the \em level \em of $\widetilde{\lambda}$, where $c$ is the canonical
central element and $\widetilde{\lambda} = \sum_{i\in I \cup \{0\}} \lambda_i \Lambda_i$ is the
affine weight.  In our work, we will often compute the 0-weight $\lambda_0
\Lambda_0$ at level 0 for a node $b$ in a classical crystal from the classical
weight $\lambda = \sum_{i\in I} \lambda_i \Lambda_i = \wt(b)$ by setting
$\langle \lambda_0 \Lambda_0 + \lambda, c \rangle = 0$ and solving for
$\lambda_0$.

Suppose that $\mathfrak{g}$ is a symmetrizable Kac--Moody algebra and let
$U_q'(\mathfrak{g})$ be the corresponding quantum algebra without derivation.
The goal of this work is to study crystals $B^{r,s}$ that correspond to 
certain finite dimensional $U_q'(\mathfrak{g})$-modules known as
Kirillov--Reshetikhin modules.  Here, $r$ is a node of the Dynkin diagram and
$s$ is a nonnegative integer.  The existence of the crystals $B^{r,s}$ that we
study follows from results in \cite{KKMMNN:1992}, while the classical
decomposition of these crystals is given in \cite{Chari:2001}.

\subsection{Tensor products of crystals} \label{ss:tensor}
Let $B_1,B_2,\dotsc,B_L$ be $U_q(\mathfrak{g})$-crystals. The
Cartesian product $B_1 \times B_2 \times \dotsm \times B_L$ has the
structure of a $U_q(\mathfrak{g})$-crystal using the so-called signature
rule. The resulting crystal is denoted
$B=B_1 \otimes B_2 \otimes\dots\otimes B_L$ and its elements
$(b_1,\dotsc,b_L)$ are written $b_1\otimes \dotsm \otimes b_L$
where $b_j\in B_j$. The reader is warned that our convention is
opposite to that of Kashiwara \cite{K:1995}. Fix $i\in I$ and
$b=b_1\otimes\dotsm\otimes b_L\in B$. The \em $i$-signature \em of $b$ is
the word consisting of the symbols $+$ and $-$ given by
\begin{equation*}
\underset{\text{$\varphi_i(b_1)$ times}}{\underbrace{-\dotsm-}}
\quad \underset{\text{$\varepsilon_i(b_1)$
times}}{\underbrace{+\dotsm+}} \,\dotsm\,
\underset{\text{$\varphi_i(b_L)$ times}}{\underbrace{-\dotsm-}}
\quad \underset{\text{$\varepsilon_i(b_L)$
times}}{\underbrace{+\dotsm+}} .
\end{equation*}
The \em reduced $i$-signature \em of $b$ is the subword of the
$i$-signature of $b$, given by the repeated removal of adjacent
symbols $+-$ (in that order); it has the form
\begin{equation*}
\underset{\text{$\varphi_i$ times}}{\underbrace{-\dotsm-}} \quad
\underset{\text{$\varepsilon_i$ times}}{\underbrace{+\dotsm+}}.
\end{equation*}
If $\varphi_i=0$ then $f_i(b)=\zero$; otherwise
\begin{equation*}
f_i(b_1\otimes\dotsm\otimes b_L)= b_1\otimes \dotsm \otimes
b_{j-1} \otimes f_i(b_j)\otimes \dots \otimes b_L
\end{equation*}
where the rightmost symbol $-$ in the reduced $i$-signature of
$b$ comes from $b_j$. Similarly, if $\varepsilon_i=0$ then
$e_i(b)=\zero$; otherwise
\begin{equation*}
e_i(b_1\otimes\dotsm\otimes b_L)= b_1\otimes \dotsm \otimes
b_{j-1} \otimes e_i(b_j)\otimes \dots \otimes b_L
\end{equation*}
where the leftmost symbol $+$ in the reduced $i$-signature of $b$
comes from $b_j$. It is not hard to verify that this defines
the structure of a $U_q(\mathfrak{g})$-crystal with
$\varphi_i(b)=\varphi_i$ and $\varepsilon_i(b)=\varepsilon_i$ in the above
notation, and weight function
\begin{equation*} 
\wt(b_1\otimes\dotsm\otimes b_L)=\sum_{j=1}^L \wt(b_j).
\end{equation*}

\subsection{Fundamental crystals for type $E_6$ and $E_7$}\label{ss:fund_cr}

Let $I = \{1, 2, 3, 4, 5, 6\}$ denote the classical index set for $E_6$.  We
number the nodes of the affine Dynkin diagram as in Figure~\ref{fig:E67}.  

\begin{figure}[h]
\begin{tabular}{ccc}
\xymatrix@-1.2pc{
                 &                  & \gn^0 \ar@{-}[d]  &                  & \\
                 &                  & \gn^2 \ar@{-}[d]  &                  & \\
\gn^1 \ar@{-}[r] & \gn^3 \ar@{-}[r] & \gn^4 \ar@{-}[r] & \gn^5 \ar@{-}[r] & \gn^6 \\
} & \ \ & 
\xymatrix@-1.2pc{
 & & & \ \ &  \\
 &                  &                  & \gn^2 \ar@{-}[d]  &                  & \\
\gn^0 \ar@{-}[r] & \gn^1 \ar@{-}[r] & \gn^3 \ar@{-}[r] & \gn^4 \ar@{-}[r] & \gn^5 \ar@{-}[r] & \gn^6 \ar@{-}[r] & \gn^7 \\
}
\end{tabular}
\caption{Affine $E_6^{(1)}$ and $E_7^{(1)}$ Dynkin diagrams}\label{fig:E67}
\end{figure}
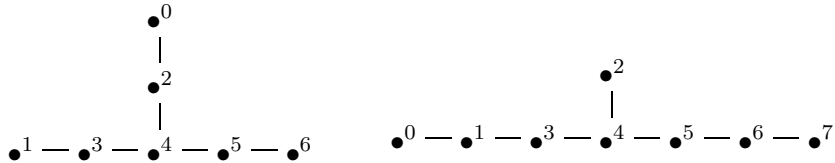

Classical highest-weight crystals $B(\lambda)$ for $E_6$ can 
be realized by the Lenart--Postnikov alcove path model described in \cite{lenart-postnikov}.
We implemented this model in \sage and have recorded the crystal $B(\Lambda_1)$
in Figure~\ref{fig:L1}.  This crystal has 27 nodes.  

To describe our labeling of the nodes, observe that all of the $i$-strings in
$B(\Lambda_1)$ have length 1 for each $i \in I$.  Therefore, the crystal admits
a transitive action of the Weyl group.  Also, it is straightforward to verify
that all of the nodes in $B(\Lambda_1)$ are determined by weight.  For our work
in Section~\ref{s:affine_structure}, we also compute the 0-weight at level 0 of
a node $b$ in any classical crystal from the classical weight as described in
Remark~\ref{r:reality}.

Thus, we label the nodes of $B(\Lambda_1)$ by weight, which is equivalent
to recording which $i$-arrows come in and out of $b$.  The $i$-arrows into $b$
are recorded with an overline to indicate that they contribute negative weight, while
the $i$-arrows out of $b$ contribute positive weight.

Using the Lenart--Postnikov alcove path model again, we can verify that
$B(\Lambda_6)$ also has 27 nodes and is dual to $B(\Lambda_1)$ in the sense
that its crystal graph is obtained from $B(\Lambda_1)$ by reversing all of the
arrows.  Reversing the arrows requires us to label the nodes of $B(\Lambda_6)$
by the weight that is the negative of the weight of the corresponding node in $B(\Lambda_1)$.
Moreover, observe that $B(\Lambda_1)$ contains no pair of nodes with weights
$\mu$, $-\mu$, respectively.  Hence, we can unambiguously label any node of
$B(\Lambda_1) \union B(\Lambda_6)$ by weight.

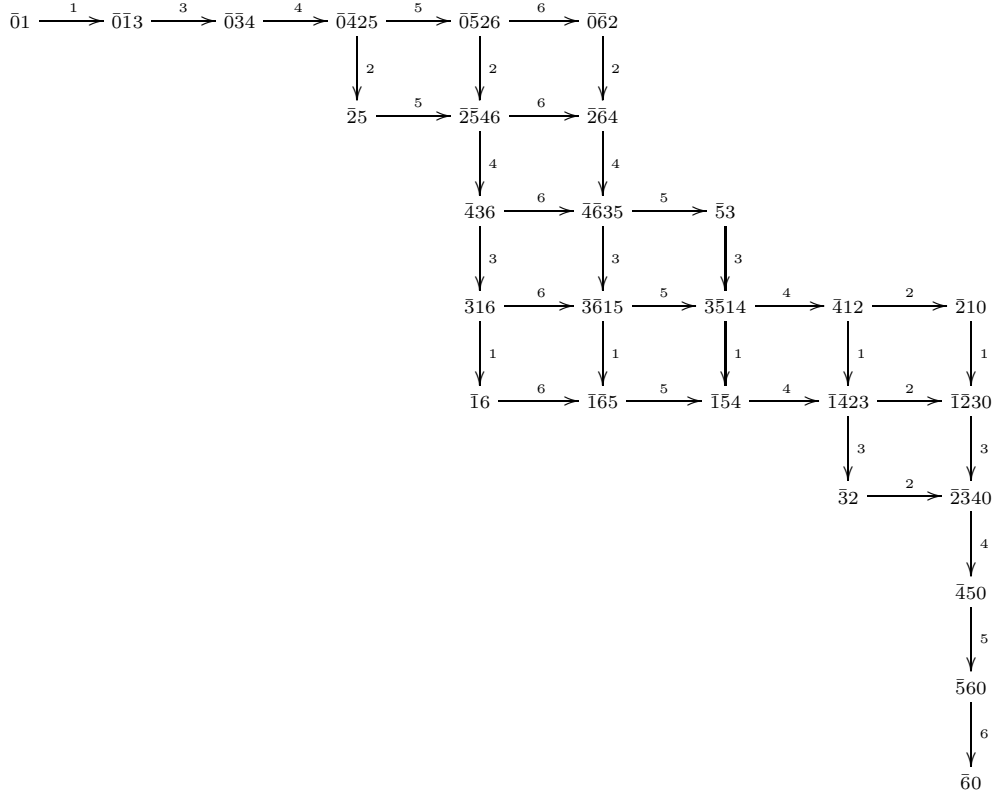
\begin{figure}[ht]
\begin{center}
\tiny{
\xymatrix{
\bar{0}1 \ar[r]^1 & \bar{0}\bar{1}3 \ar[r]^3 & \bar{0}\bar{3}4 \ar[r]^4 & \bar{0}\bar{4}25 \ar[r]^5 \ar[d]^2 & \bar{0}\bar{5}26 \ar[r]^6 \ar[d]^2 & \bar{0}\bar{6}2 \ar[d]^2 \\
                  &                          &                          & \bar{2}5 \ar[r]^5                  & \bar{2}\bar{5}46 \ar[r]^6 \ar[d]^4 & \bar{2}\bar{6}4 \ar[d]^4 \\
                  &                          &                          &                                    & \bar{4}36 \ar[r]^6 \ar[d]^3        & \bar{4}\bar{6}35 \ar[r]^5 \ar[d]^3 & \bar{5}3 \ar[d]^3  \\
                  &                          &                          &                                    & \bar{3}16 \ar[r]^6 \ar[d]^1        & \bar{3}\bar{6}15 \ar[r]^5 \ar[d]^1 & \bar{3}\bar{5}14 \ar[r]^4 \ar[d]^1 & \bar{4}12 \ar[r]^2 \ar[d]^1 & \bar{2}10 \ar[d]^1  \\
                  &                          &                          &                                    & \bar{1}6 \ar[r]^6                  & \bar{1}\bar{6}5 \ar[r]^5           & \bar{1}\bar{5}4 \ar[r]^4           & \bar{1}\bar{4}23 \ar[r]^2 \ar[d]^3 & \bar{1}\bar{2}30 \ar[d]^3  \\
                  &                          &                          &                                    &                                    &                                    &                                    & \bar{3}2 \ar[r]^2           & \bar{2}\bar{3}40 \ar[d]^4  \\
                  &                          &                          &                                    &                                    &                                    &                                    &                             & \bar{4}50 \ar[d]^5  \\
                  &                          &                          &                                    &                                    &                                    &                                    &                             & \bar{5}60 \ar[d]^6  \\
                  &                          &                          &                                    &                                    &                                    &                                    &                             & \bar{6}0  \\
} }
\end{center}
\caption{Crystal graph for $B(\Lambda_1)$ of type $E_6$ \label{fig:L1}}
\end{figure}

It is straightforward to show using characters that every classical
highest-weight representation $B(\Lambda_i)$ for $i \in I$ can be realized as a
component of some tensor product of $B(\Lambda_1)$ and $B(\Lambda_6)$ factors.
On the level of crystals, the tensor products $B(\Lambda_1)^{\otimes k}$,
$B(\Lambda_6)^{\otimes k}$ and $B(\Lambda_6) \otimes B(\Lambda_1)$ are defined
for all $k$ by the tensor product rule of Section~\ref{ss:tensor}.  
Therefore, we can realize the other classical fundamental
crystals $B(\Lambda_i)$ as shown in Table~\ref{t:fund_real_E6}.  There are
additional realizations for these crystals obtained by dualizing.

\begin{table}[ht]
\renewcommand{\arraystretch}{1.2}
\caption{Fundamental realizations for $E_6$}\label{t:fund_real_E6}
\begin{tabular}{|c|r|l|c|}
    \hline
                                & Generator                    & in                                                                    & Dimension\\
    \hline \hline
$B(\Lambda_2)$ & $2\bar{1}\bar{0} \otimes \bar{0}1$ & $B(\Lambda_6) \otimes B(\Lambda_1)$ & 78\\
    \hline
$B(\Lambda_3)$ & $\bar{0}\bar{1}3 \otimes \bar{0}1$ & $B(\Lambda_1)^{\otimes 2}$ & 351\\
    \hline
$B(\Lambda_4)$ & $\bar{0}\bar{3}4 \otimes \bar{0}\bar{1}3 \otimes \bar{0}1$ & $B(\Lambda_1)^{\otimes 3}$ & 2925\\
    \hline
$B(\Lambda_5)$ & $5\bar{6}\bar{0} \otimes 6\bar{0}$ & $B(\Lambda_6)^{\otimes 2}$ & 351\\
    \hline
\end{tabular}
\end{table}

\begin{figure}[hp]
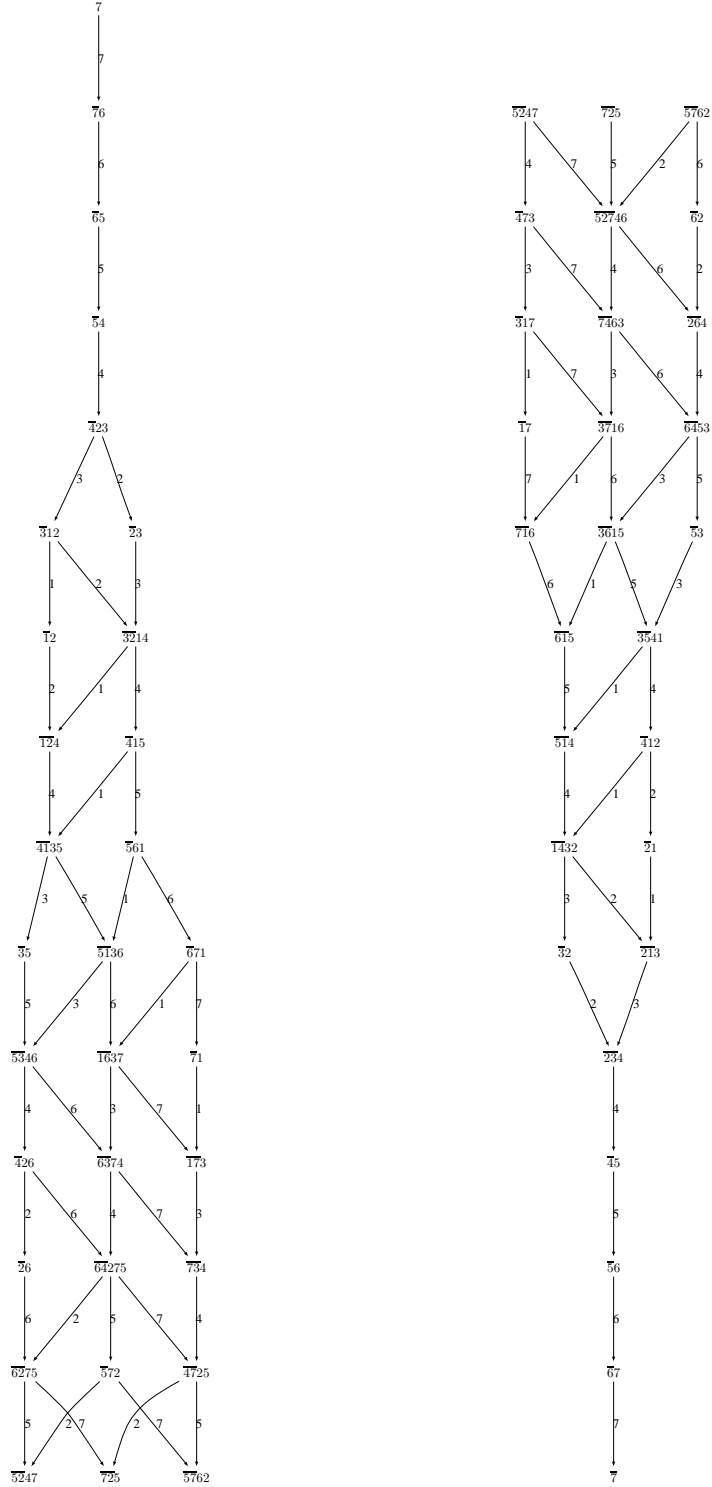

\include{typeE}
\caption{$B(\Lambda_7)$ of type $E_7$} \label{fig:BLa7}
\end{figure}

The Dynkin diagram of type $E_7^{(1)}$ is shown in Figure~\ref{fig:E67}.  The
highest weight crystal $B(\Lambda_7)$ has 56 nodes and these nodes all have
distinct weights (see Figure~\ref{fig:BLa7}).  Also, $\varphi_i(b) \leq 1$ and
$\varepsilon_i(b) \leq 1$ for all $i \in \{1, 2, \ldots, 7\}$ and $b \in
B(\Lambda_7)$.  Using character calculations, we can show that every classical
highest-weight representation $B(\Lambda_i)$ appears in some tensor product of
$B(\Lambda_7)$ factors.  In Table~\ref{t:fund_real_E7}, we display realizations
for all of the classical fundamental crystals $B(\Lambda_i)$ in type $E_7$.

Green~\cite{green:2007, green:2008} has another construction of the 27-dimensional
crystals $B(\Lambda_1)$ and $B(\Lambda_6)$ of type $E_6$, and the 56-dimensional
crystal $B(\Lambda_7)$ of type $E_7$ in terms of full heaps, and also gives the
connection of the fundamental $E_6$ crystals with the 27 lines on a cubic surface.
A Littlewood-Richardson rule for type $E_6$ was given in~\cite{hoshino:2007} using
polyhedral realizations of crystal bases.

\begin{table}[ht]
\renewcommand{\arraystretch}{1.2}
\caption{Fundamental realizations for $E_7$}\label{t:fund_real_E7}
\begin{tabular}{|c|r|l|c|}
    \hline
                                & Generator                    & in                                                & Dimension\\
    \hline \hline
$B(\Lambda_1)$ & $\bar{0}\bar{7}1 \otimes \bar{0}7$ & $B(\Lambda_7)^{\otimes 2}$ & 133\\
    \hline
$B(\Lambda_2)$ & $\bar{1}2 \otimes \bar{0}\bar{7}1 \otimes \bar{0}7$ & $B(\Lambda_7)^{\otimes 3}$ & 912 \\
    \hline
$B(\Lambda_3)$ & $\bar{0}\bar{2}3 \otimes \bar{1}2 \otimes \bar{0}\bar{7}1 \otimes \bar{0}7$ & $B(\Lambda_7)^{\otimes 4}$
   & 8645  \\
    \hline
$B(\Lambda_4)$ & $\bar{0}\bar{5}4 \otimes \bar{0}\bar{6}5 \otimes \bar{0}\bar{7}6 \otimes \bar{0}7$ & $B(\Lambda_7)^{\otimes 4}$
    & 365750 \\
    \hline
$B(\Lambda_5)$ & $\bar{0}\bar{6}5 \otimes \bar{0}\bar{7}6 \otimes \bar{0}7$ & $B(\Lambda_7)^{\otimes 3}$
    & 27664 \\
    \hline
$B(\Lambda_6)$ & $\bar{0}\bar{7}6 \otimes \bar{0}7$ & $B(\Lambda_7)^{\otimes 2}$
    & 1539 \\
    \hline
$B(\Lambda_7)$ & $\bar{0}7$ & $B(\Lambda_7)$
    & 56 \\
    \hline

\end{tabular}
\end{table}


\bigskip
\subsection{Generalized tableaux} \label{ss:gen_tab}

In this section, we describe how to realize the crystal $B(\Lambda_{i_1} +
\Lambda_{i_2} + \cdots + \Lambda_{i_k})$ inside the tensor product
$B(\Lambda_{i_1}) \otimes B(\Lambda_{i_2}) \otimes \cdots \otimes
B(\Lambda_{i_k})$, where the $\Lambda_i$ are all fundamental, or more generally
dominant weights.  Our arguments use only abstract crystal properties, so the
results in this section apply to any finite type.

If $b$ is the unique highest weight node in $B(\lambda)$ and $c$ is the unique
highest weight node in $B(\mu)$, then $B(\lambda + \mu)$ is generated by $b
\otimes c \in B(\lambda) \otimes B(\mu)$.  Iterating this procedure provides a
recursive description of any highest-weight crystal embedded in a tensor product
of crystals.  Our goal is to give a non-recursive description of the nodes of
$B( \Lambda_{i_1} + \Lambda_{i_2} + \cdots + \Lambda_{i_k} )$ for any collection
of fundamental weights $\Lambda_i$.

For an ordered set of dominant weights $(\mu_1, \mu_2, \ldots, \mu_k)$ and
for each permutation $w$ in the symmetric group $S_k$, define
\[ B_w(\mu_1, \ldots, \mu_k) = B(\mu_{w(1)}) \otimes B(\mu_{w(2)}) \otimes \cdots \otimes B(\mu_{w(k)}) \]
so $B_e(\mu_1, \ldots, \mu_k)$ is $B(\mu_1) \otimes \cdots \otimes B(\mu_k)$ where $e \in S_k$ is the
identity.  

\begin{definition}\label{d:pwi}
Let $(\mu_1, \mu_2, \ldots, \mu_k)$ be dominant weights.  Then, we say that 
\[ b_1 \otimes b_2 \otimes \cdots \otimes b_k \in B(\mu_{1}) \otimes B(\mu_{2}) \otimes \cdots \otimes B(\mu_k) \]
is \em pairwise weakly increasing \em if 
\[ b_j \otimes b_{j+1} \in B(\mu_j + \mu_{j+1}) \subset B(\mu_j) \otimes B(\mu_{j+1}) \]
for each $1 \leq j < k$.
\end{definition}

Next, we fix an isomorphism of crystals
\[ \Phi_{w}^{(\mu_1, \ldots, \mu_k)} : B_w(\mu_1, \ldots, \mu_k) \rightarrow B_e(\mu_1, \ldots, \mu_k) \]
for each $w \in S_k$.  Observe that each choice of $\Phi_w^{(\mu_1, \ldots,
\mu_k)}$ corresponds to a choice for the image of each of the highest-weight
nodes in $B_w(\mu_1, \ldots, \mu_k)$.  

Let $b_j^{\ast}$ denote the unique highest weight node of the $j$th factor
$B(\mu_j)$.  Since we are fixing the dominant weights $(\mu_1, \ldots,
\mu_k)$, we will sometimes drop the notation $(\mu_1, \ldots, \mu_k)$ from $B_w$
and $\Phi_w$ in the proofs below.

\begin{definition}
Let $w$ be a permutation that fixes $\{1, 2, \ldots, j\}$.  We say that
$\Phi_{w}^{(\mu_1, \ldots, \mu_k)}$ is a \em lazy isomorphism \em if the image
of every highest weight node of the form 
\[
b_1 \otimes b_2 \otimes \cdots \otimes b_j \otimes b_{j+1}^{\ast} \otimes \cdots
\otimes b_{k}^{\ast}
\]
under $\Phi_{w}^{(\mu_1, \ldots, \mu_k)}$ is equal to
\[
b_1 \otimes b_2 \otimes \cdots \otimes b_j \otimes b_{w^{-1}(j+1)}^{\ast} \otimes \cdots
\otimes b_{w^{-1}(k)}^{\ast}\; .
\]
\end{definition}

We want to choose our isomorphisms $\Phi_{w}^{(\mu_1, \ldots, \mu_k)}$ to be
lazy, but we will see in the course of the proofs that our results do not
otherwise depend upon the choice of $\Phi_w^{(\mu_1, \ldots, \mu_k)}$.  

\begin{definition} \label{d:weak_inc}
Let $T$ be any subset of $S_k$, and $\{\Phi_{w}^{(\mu_1, \ldots, \mu_k)}\}_{w
\in T}$ be a collection of lazy isomorphisms.  We define $I^{(\mu_1, \ldots,
\mu_k)}(T)$ to be
\[ \bigcap_{w \in T} \Phi_w^{(\mu_1, \ldots,
\mu_k)} ( \{ \text{pairwise weakly increasing nodes of $B_w(\mu_1, \ldots,
\mu_k)$ } \} ) \subset B_e(\mu_1, \ldots, \mu_k). \]
\end{definition}

\begin{proposition}\label{p:wi_closed}
Let $T$ be any subset of $S_k$.  Then, whenever $b \in I^{(\mu_1, \ldots,
\mu_k)}(T)$ we have $e_i(b), f_i(b) \in I^{(\mu_1, \ldots, \mu_k)}(T)$.
\end{proposition}
\begin{proof}
We first claim that the crystal operators $e_i$ and $f_i$ preserve the pairwise
weakly increasing condition in any tensor product of highest weight crystals.
Let 
\[ b = b_1 \otimes b_2 \otimes \cdots \otimes b_k \]
be a pairwise weakly increasing node in $B = B(\mu_{1}) \otimes \cdots \otimes
B(\mu_k)$.

We need to show that $e_i(b)$ is pairwise weakly increasing. Suppose that $e_i$
acts on the $j$-th tensor factor in $b$, that is, $e_i(b) = b_1 \otimes \cdots
\otimes e_i(b_j) \otimes \cdots \otimes b_k$. Hence it suffices to show that
$b_{j-1} \otimes e_i(b_j)\in B(\mu_{j-1}+\mu_j)$ and $e_i(b_j) \otimes
b_{j+1}\in B(\mu_j + \mu_{j+1})$.  Since $e_i$ acts on $b_j$ in $b$, in the
tensor product rule the leftmost unbracketed $+$ is associated to $b_j$.  This
means that any $+$ from $b_{j-1}$ must be bracketed with a $-$ from $b_j$.  But
then $e_i(b_{j-1}\otimes b_j) =b_{j-1} \otimes e_i(b_j)\in B(\mu_{j-1}+\mu_j)$.
Similarly, since $e_i$ acts on $b_j$, not all $+$ in $b_j$ are bracketed with
$-$ in $b_{j+1}\otimes \cdots \otimes b_k$. But therefore, also not all $+$ in
$b_j$ are bracketed with $-$ in $b_{j+1}$ and hence $e_i(b_j\otimes b_{j+1}) =
e_i(b_j) \otimes b_{j+1} \in B(\mu_j+\mu_{j+1})$.  The arguments for $f_i$ are
analogous.

Next, suppose that $b \in I^{(\mu_1, \ldots, \mu_k)}(T) \subset B_e$.  Then, for
all $w \in S_k$ we have $\Phi_w^{-1}(b)$ is pairwise weakly increasing in $B_w$.
By the argument above, we then have that $e_i( \Phi_w^{-1}(b) )$ is pairwise
weakly increasing in $B_w$.  Since $\Phi_w$ is an isomorphism, it commutes with
$e_i$, so $\Phi_w^{-1}( e_i (b) )$ is pairwise weakly increasing in $B_w$ for
all $w \in S_k$.  Hence, $e_i(b) \in I^{(\mu_1, \ldots, \mu_k)}(T)$.  The
arguments for $f_i$ are analogous.
\end{proof}

\begin{corollary}
For any subset $T$ of $S_k$, we have that $I^{(\mu_1, \ldots, \mu_k)}(T)$ is a direct sum of highest
weight crystals $\bigoplus_{\lambda} B(\lambda)$ for some collection of weights $\lambda$.
\end{corollary}
\begin{proof}
Proposition~\ref{p:wi_closed} implies that whenever $b \in I^{(\mu_1, \ldots,
\mu_k)}(T)$, the entire connected component of the crystal graph containing $b$
is in $I^{(\mu_1, \ldots, \mu_k)}(T)$.
\end{proof}

\begin{theorem}\label{t:tableaux}
Fix a sequence $(\mu_1, \ldots, \mu_k)$ of dominant weights.  Then,
\[ I^{(\mu_1, \ldots, \mu_k)}(S_k) \cong B(\mu_1 + \mu_2 + \ldots + \mu_k). \]
\end{theorem}
\begin{proof}
Let $b_j^{\ast}$ be the unique highest weight node of $B_j$ with highest weight
$\mu_j$ for each $j = 1, \ldots, k$.  Then $b^{\ast} = b_1^{\ast} \otimes
b_2^{\ast} \otimes \cdots \otimes b_k^{\ast}$ generates $B(\mu_1 + \ldots +
\mu_k)$ and this node lies in $I^{(\mu_1, \ldots, \mu_k)}(S_k)$.

Suppose there exists another highest weight node in $I^{(\mu_1, \ldots,
\mu_k)}(S_k)$.  Then, at least one of the factors $b_j$ must have
$\varepsilon_i(b_j) > 0$ for some $i$.  Choose $j$ to be the rightmost factor
having $\varepsilon_i(b_j) > 0$ for some $i \in I$.  Then fix some choice of $i$
such that $\varepsilon_i(b_j) > 0$.  Our highest weight node has the form
\[ b = b_1 \otimes \cdots \otimes b_j \otimes b_{j+1}^{\ast} \otimes \cdots \otimes
b_{k}^{\ast}. \]
In particular, $j < k$ since any rightmost factor of a highest weight tensor
product must be highest weight.

Since $b$ is highest weight, we have that all $+$ entries for factor $b_j$ are
canceled by $-$ entries lying to the right in the $i$-signature for the tensor
product rule.  Suppose that $b_{j'}$ is the leftmost factor for which a $-$
cancels a $+$ from $b_j$ in the $i$-signature.  Let $w$ be the permutation that
interchanges factors $j+1$ and $j'$.  Then, by our choice of $\Phi_w$ we have
that $\Phi_w^{-1}(b)$ is obtained from $b$ just by interchanging the factors
$b_{j+1}^{\ast}$ and $b_{j'}^{\ast}$.

Hence, we have that $\Phi_w^{-1}(b)$ in $B_w$ has an adjacent $+$/$-$ pair on
factors $j, j+1$.  Since this pair is part of a pairwise weakly increasing
element, there must exist a sequence of $e_{i'}$ operations that brings $b_j
\otimes b_{j'}^{\ast}$ to $b_j^{\ast} \otimes b_{j'}^{\ast}$.  However, $e_{i'}$
can only operate on the first tensor factor in this pair because $b_{j'}^{\ast}$
is already highest weight.  Moreover, we have that $\varepsilon_i$ of the
first factor and $\varphi_i$ of the second factor are both positive.  This
remains true regardless of how we apply $e_{i'}$ operations where $i \neq i'$ by
\cite[Axiom (P4)]{stembridge:2003}.  We can potentially apply the $e_i$
operation $\max\{\varepsilon_i(b_j) - \varphi_i(b_{j'}^{\ast}),0\}$ times, but
since $\varphi_i(b_{j'}^{\ast}) > 0$, we have that $\varepsilon_i$ of the
first factor will always remain positive.  Hence, we can never reach $b_j^{\ast}
\otimes b_{j'}^{\ast}$, a contradiction.  

Thus, $b^{\ast}$ is the unique highest weight node of $I^{(\mu_1, \ldots, \mu_k)}(S_k)$.
\end{proof}

\begin{remark}
The condition that there is a unique highest weight element that we used in the
proof of Theorem~\ref{t:tableaux} is equivalent to the hypothesis of
\cite[Proposition 2.2.1]{KN} from which the desired conclusion also follows. 
\end{remark}

\begin{remark}
Observe that only a finite constant amount of data is ever required to check the
pairwise weakly increasing condition, regardless of how large the number of
tensor factors $k$ is.  Theorem~\ref{t:tableaux} and its refinements will allow
us to formulate arguments that apply to all highest-weight crystals
simultaneously.

When we are considering a specific highest-weight crystal, it may be
computationally easier to generate $B(\mu_1 + \cdots + \mu_k)$ by simply
applying $f_i$ operations to the highest-weight node in all possible ways.
\end{remark}

We will say that any node of $I^{(\mu_1, \ldots, \mu_k)}(S_k)$ is \em weakly
increasing\em.  It turns out that we can often take $T$ to be much smaller than
$S_n$ by starting with $T = \{e\}$ and adding permutations to $T$ until
$I^{(\mu_1,\ldots,\mu_k)}(T)$ contains a unique highest weight node.  In
particular, the next result shows that we can take $T = \{e\}$ when we are
considering a linear combination of two distinct fundamental weights.

\begin{lemma}\label{l:weakly_increasing}
Let $\Lambda_{i_1}$ and $\Lambda_{i_2}$ be distinct fundamental weights, and
$k_1, k_2 \in \mathbb{Z}_{\geq 0}$ with $k = k_1 + k_2$.
Then, the nodes of
\[ B(k_1 \Lambda_{i_1} + k_2 \Lambda_{i_2}) \subset B(\Lambda_{i_1})^{\otimes
k_1} \otimes B(\Lambda_{i_2})^{\otimes k_2} \]
are precisely the pairwise weakly increasing tensor products $b_1 \otimes b_2
\otimes \cdots \otimes b_k$ of $B(\Lambda_{i_1})^{\otimes k_1} \otimes
B(\Lambda_{i_2})^{\otimes k_2}$.
\end{lemma}
\begin{proof}
We order the fundamental weights as $(\Lambda_{i_1}, \ldots, \Lambda_{i_1}, \Lambda_{i_2},
\ldots, \Lambda_{i_2})$ and apply the same argument 
as in the proof of Theorem~\ref{t:tableaux} to see that any highest weight node
in $I^{(\Lambda_{i_1}, \ldots, \Lambda_{i_1}, \Lambda_{i_2}, \ldots, \Lambda_{i_2})}(\{e\})$
must be of the form
\[ b_1 \otimes \cdots \otimes b_{k_1 - 1} \otimes b_{k_1}^{\ast} \otimes b_{k_1
+ 1}^{\ast} \otimes \cdots \otimes b_{k}^{\ast}. \]
In this case, it is never necessary to apply $\Phi_w$ to reorder the factors
because all of the factors to the right of factor $k_1$ must be the same.

Next, we let $j = k_1 - 1$.  We have that $b_{j+1} = b_{j+1}^{\ast}$ and we work
by downward induction to argue that $b_j$ must be $b_j^{\ast}$.  This follows
because due to the pairwise weak increasing condition there exists a sequence of
$e_i$ that takes $b_j \otimes b_{j+1}^{\ast}$ to $b_j^{\ast} \otimes
b_{j+1}^{\ast}$.  The highest weight node of the fundamental crystal
$B(\Lambda_{i_1})$ has a unique $i_1$-arrow.  If $b_j \neq b_j^{\ast}$ then we
could never traverse this edge because in the $i_1$-signature any $+$ would be
canceled by a $-$ from $b_{j+1}^{\ast}$.  Hence, $b_j = b_j^{\ast}$, and the
induction continues.

Thus, there is a unique highest-weight node in $I^{(\Lambda_{i_1}, \ldots,
\Lambda_{i_1}, \Lambda_{i_2}, \ldots, \Lambda_{i_2})}(\{e\})$.
\end{proof}

All of the crystals in our work have classical decompositions that have been
given by Chari \cite{Chari:2001}.  These crystals satisfy the requirement of
Lemma~\ref{l:weakly_increasing} that at most two fundamental weights appear.  On
the other hand, Example~\ref{ce:KN} shows that no ordering of the factors in
$B(\Lambda_2) \otimes B(\Lambda_1) \otimes B(\Lambda_6)$ in type $E_6$ admits an
analogous weakly increasing condition that is defined using only pairwise
comparisons.

\begin{example}\label{ce:KN}
Observe that each of the following nodes in type $E_6$ is a counterexample to
the condition required in \cite[Proposition 2.2.1]{KN}.  Each of the given nodes
is highest weight, and pairwise weakly increasing, but none of the nodes
correspond to the highest weight node of $B(\Lambda_1+\Lambda_6+\Lambda_2)$.

\begin{eqnarray*}
(3\bar{1}\bar{6} \otimes 1) \otimes u_1 \otimes u_6 \in B(\Lambda_2) \otimes B(\Lambda_1) \otimes B(\Lambda_6)  \\
(5\bar{3} \otimes \bar{1}3) \otimes u_6 \otimes u_1 \in B(\Lambda_2) \otimes B(\Lambda_6) \otimes B(\Lambda_1)  \\
\bar{2}5 \otimes u_6 \otimes u_2 \in B(\Lambda_1) \otimes B(\Lambda_6) \otimes B(\Lambda_2) \\
\bar{6}2 \otimes u_2 \otimes u_6 \in B(\Lambda_1) \otimes B(\Lambda_2) \otimes B(\Lambda_6) \\
\bar{2}3 \otimes u_1 \otimes u_2 \in B(\Lambda_6) \otimes B(\Lambda_1) \otimes B(\Lambda_2) \\
2\bar{1} \otimes u_2 \otimes u_1 \in B(\Lambda_6) \otimes B(\Lambda_2) \otimes B(\Lambda_1) \\
\end{eqnarray*}

Here, $u_i$ is the highest weight node of $B(\Lambda_i)$.  Hence, it is not possible
to obtain a pairwise weakly increasing condition that characterizes the nodes of
$B(\Lambda_1+\Lambda_6+\Lambda_2)$.
\end{example}

\begin{remark}
In standard monomial theory~\cite{LS:1986}, the condition that a tensor product of
basis elements lies in $B(\lambda+\mu)$ can also be formulated as a comparison of the lift
of these elements in Bruhat order~\cite{littelmann:1996}. For several tensor factors, one needs 
to compare simultaneous lifts.
\end{remark}

We now restrict to type $E_6$.  Lemma~\ref{l:weakly_increasing} implies that we
have a non-recursive description of all $B(k \Lambda_i)$ determined by the
finite information in $B(2 \Lambda_i)$.  In the case of particular fundamental
representations, we can be more specific about how to test for the weakly
increasing condition.

\begin{proposition} \label{p:L1}
We have that $b_1 \otimes b_2 \in B(2 \Lambda_1) \subset B(\Lambda_1)^{\otimes
2}$ if and only if $b_2$ can be reached from $b_1$ by a sequence of $f_i$
operations in $B(\Lambda_1)$.
\end{proposition}
\begin{proof}
This is a finite computation on $B(2 \Lambda_1)$.
\end{proof}

The crystal graph for $B(\Lambda_1)$ of Figure~\ref{fig:L1} can be viewed as a poset.
Then Proposition~\ref{p:L1} implies in particular that incomparable pairs in $B(\Lambda_1)$ are 
not weakly increasing.

There are 78 nodes in $B(\Lambda_2)$.  We construct $B(\Lambda_2)$ as the
highest weight crystal graph generated by $2\bar{1}\bar{0} \otimes \bar{0}1$
inside $B(\Lambda_6) \otimes B(\Lambda_1)$.  Note that we only need to use the
nodes in the ``top half'' of Figure~\ref{fig:L1} and their duals.  There are
2430 nodes in $B(2 \Lambda_2)$.  

\begin{proposition}\label{p:L2wi}
We have that 
\[ (b_1 \otimes c_1) \otimes (b_2 \otimes c_2) \in B(2 \Lambda_2) \subset (B(\Lambda_6) \otimes B(\Lambda_1))^{\otimes 2} \]
if and only if
\begin{enumerate}
    \item[(1)]  $b_2$ can be reached from $b_1$ by $f_i$ operations in
        $B(\Lambda_6)$, and $c_2$ can be reached from $c_1$ by $f_i$ operations in $B(\Lambda_1)$, and
    \item[(2)]  Whenever $c_1$ is dual to $b_2$, we have that there is a path of
        $f_i$ operations from $(b_1 \otimes c_1)$ to $(b_2 \otimes c_2)$ of length at least 1 (so in particular, the elements are not equal) in $B(\Lambda_2)$.
\end{enumerate}
\end{proposition}
\begin{proof}
This is a finite computation on $B(2 \Lambda_2)$.
\end{proof}


\bigskip
\section{Affine structure}\label{s:affine_structure}

In this section, we study the affine crystals of type $E_6^{(1)}$.
We introduce the method of promotion to obtain a combinatorial affine crystal structure in
Section~\ref{ss:promotion} and the notion of composition graphs in Section~\ref{ss:CG}.
It is shown in Theorem~\ref{t:main} that
order three twisted isomorphisms yield regular affine crystals. This is used to
construct $B^{r,s}$ of type $E_6^{(1)}$ for the minuscule nodes $r=1,6$ in
Section~\ref{ss:L1} and the adjoint node $r=2$ in Section~\ref{ss:L2}.  In
Section~\ref{ss:E7} we present conjectures for $B^{1,s}$ of type $E_7^{(1)}$.

\subsection{Combinatorial affine crystals and twisted isomorphisms}
\label{ss:promotion}

The following concept is fundamental to this work.

\begin{definition}\label{d:ti}
Let $\widetilde{C}$ be an affine Dynkin diagram and $C$ the associated finite Dynkin
diagram (obtained by removing node 0) with index set $I$.
Let $\p$ be an automorphism of $\widetilde{C}$, and $B$ be a classical crystal of type $C$.  
We say that \em $\p$ induces a twisted isomorphism of crystals \em if
there exists a bijection of crystals $p: B \cup \{\zero\} \rightarrow B' \cup \{\zero\}$
satisfying 
\begin{equation}\label{e:zero}
p(b) = \zero \text{ if and only if $b = \zero$, and }
\end{equation}
\begin{equation}\label{e:twist}
p \circ f_i (b) = f_{\p(i)} \circ p (b) \text{ and } p \circ e_i (b) = e_{\p(i)} \circ p (b)
\end{equation}
for all $i \in I \setminus \{\p^{-1}(0)\}$ and all $b \in B$. 

We frequently abuse notation and denote $B'$ by $p(B)$ even though the
isomorphism $p: B \rightarrow p(B)$ may not be unique.  

If we are given two classical crystals $B$ and $B'$, and there exists a Dynkin
diagram automorphism $\p$ that induces a twisted isomorphism between $B$ and
$B'$, then we say that $B$ and $B'$ are \em twisted-isomorphic\em.
\end{definition}

\begin{definition}
Let $B$ be a crystal with index set $I$. Then $B$ is called \emph{regular} if
for any 2-subset $J\subset I$, we have that the restriction of $B$ to its $J$-arrows
is a classical rank two crystal.
\end{definition}

\begin{definition}\label{d:cas}
Let $B$ be a classical crystal with index set $I$.  Suppose $\widetilde{B}$ is a labeled
directed graph on the same nodes as $B$ and with the same $I$-arrows, but with
an additional set of 0-arrows.  If $\widetilde{B}$ is regular, then we say that 
$\widetilde{B}$ is a \em combinatorial affine structure \em for $B$.  
\end{definition}

\begin{remark}\label{r:reality}
Although we do not assume that $\widetilde{B}$ is a crystal graph for a
$U_{q}'(\mathfrak{g})$-module, Kashiwara \cite{K:2002, K:2005} has shown that
the crystals of such modules must be regular and have weights at level 0.
Therefore, we will compute the 0-weight $\lambda_0 \Lambda_0$ of the nodes $b$ in a classical 
crystal from the classical weight $\lambda = \sum_{i\in I} \lambda_i \Lambda_i = \wt(b)$ using 
the formula given in Equation~(\ref{e:level}).
\end{remark}

\begin{remark}\label{r:consequences}
Here are some consequences of Definitions~\ref{d:ti} and \ref{d:cas}.

\begin{enumerate}
\item[(1)]  Any crystal $p(B)$ induced by $\p$ is just a classical crystal that
    is isomorphic to $B$ up to relabeling.  In particular, any graph
    automorphism $\p$ induces at least one twisted isomorphism $p$: If we view
    $B$ as an edge-labeled directed graph, the image of $p$ is given on the same
    nodes as $B$ by relabeling all of the arrows according to $\p$.  On the
    other hand, it is important to emphasize that there is no canonical labeling
    for the nodes of $p(B)$.  Also, some crystal graphs may have additional
    symmetry which lead to multiple twisted isomorphisms of crystals associated
    with a single graph automorphism $\p$.
\item[(2)]  For $b \in B$, we have $\varphi(p(b)) = \sum_{i\in I}
    \varphi_{\p^{-1}(i)}(b)\Lambda_i$ and $\varepsilon(p(b)) = \sum_{i\in I}
    \varepsilon_{\p^{-1}(i)}(b)\Lambda_i$.
    In addition, we can compute the 0-weight of
    any node in $B$ by Remark~\ref{r:reality}.  Therefore, $\p$ permutes all of
    the affine weights, in the sense that 
    \[ \wt_{i}(b) = \wt_{\p(i)}(p(b)) \text{ \ \ \ \ for all $b \in B$ and $i \in I \cup \{0\}$ }. \]
\item[(3)]  Since the node $\p(0)$ becomes the affine node in $p(B)$, it is
    sometimes possible to define a combinatorial affine structure for $B$ ``by
    promotion.''  Namely, we define $f_0$ on $B$ to be $p^{-1} \circ f_{\p(0)} \circ p$.
    Note that in order for this to succeed, we must take the additional step of
    identifying the image $p(B)$ with a canonically labeled classical crystal so
    that we can infer the $f_{\p(0)}$ edges.
\end{enumerate}
\end{remark}

\begin{example}
The $E_6$ Dynkin diagram automorphism of order two that interchanges nodes 1 and 6
induces the dual map between $B(\Lambda_1)$ and $B(\Lambda_6)$.  
\end{example}

\begin{example}
Let $\p$ be the unique $E_6^{(1)}$ Dynkin diagram automorphism of order three
sending node 0 to 1.  There is no twisted isomorphism of $B(\Lambda_2)$ to
itself that is induced by $\p$.  To see this, consider the six nodes of weight 0
inside $B(\Lambda_2)$.  Observe that there is precisely one node of weight 0
lying in the center of an $i$-string for each $i \in \{1, 2, \ldots, 6\}$.  The
twisted isomorphism $p$ must send the node lying in the middle of a 6-string to
one which lies only in the middle of a 0-string and is not connected to any
other classical edges by Equation~(\ref{e:twist}).  But no such node exists in
$B(\Lambda_2)$. This is in agreement with~\eqref{eq:adjoint decomp} that an affine 
structure for the adjoint node exists on $B(\Lambda_2) \oplus B(0)$.
\end{example}

\begin{example}\label{e:prom2}
Consider the crystal $B$ of type $A_1$ shown below.
\[
\xymatrix@-1pc{ & \gn^a \ar[d]^1 \\
 B = & \gn^b \ar[d]^1 & \gn^c \\
 & \gn^d \\ } \ \ \hspace{1in} \ \ \xymatrix@-1pc{ & \gn^{a'} \ar[d]^0 \\
p(B) = & \gn^{b'} \ar[d]^0 & \gn^{c'} \\
 & \gn^{d'} \\ } 
\]
The only nontrivial graph automorphism $\p$ of the affine Dynkin diagram of
type $A_1^{(1)}$ interchanges 0 and 1, which induces $p(B)$ as shown.  However,
constructing an affine structure on $B$ by promotion requires choosing another
map from $p(B)$ back to $B$.

By considering the level-0 weight, we must identify $a'$ with $d$ as well as
$d'$ with $a$.  Since there is no restriction on $\varphi_0(b)$ nor
$\varepsilon_0(b)$ for $b \in \{b,c\}$ from the given data, the other two nodes
are undetermined.  Hence, there are two identifications which give rise to
distinct 0-arrows for $B$.
\[ 
\xymatrix@-0.5pc{ & \gn^a \ar@/_/[d]_1 \\
 & \gn^b \ar@/_/[d]_1 \ar@/_/[u]_0 & \gn^c \\
 & \gn^d \ar@/_/[u]_0\\ }
 \hspace{1in}
\xymatrix@-0.5pc{ & \gn^a \ar[d]^1 \\
 & \gn^b \ar[d]^1 & \gn^c \ar@/_/[ul]_0 \\
 & \gn^d \ar@/_/[ur]_0\\ }
\]
This example shows how twisted isomorphisms of order two can give rise to
multiple affine structures.
\end{example}

The Dynkin diagram of $E_6^{(1)}$ has an automorphism of order three that we can use to
construct combinatorial affine structures by promotion.

\begin{theorem}\label{t:main}
Let $B$ be a classical $E_6$ crystal.  Suppose there exists a bijection $p : B
\rightarrow B$ that is a twisted isomorphism satisfying $p \circ f_1 = f_6 \circ p$, and
suppose that $p$ has order three.  Then, there exists a combinatorial affine
structure on $B$.  This structure is given by defining $f_0$ to be $p^2 \circ f_1 \circ p$.
\end{theorem}
\begin{proof}
If we apply $p$ on the left and right of $p f_1 = f_6 p$, we obtain $p p f_1 p =
p f_6 p p$.  Since $p$ has order three, this is
\begin{equation}\label{e:as}
    p^{-1} f_1 p = p f_6 p^{-1}.
\end{equation}
Because $p$ is a bijection on $B$, we may define 0-arrows on $B$ by the map
$p^{-1} f_1 p$.  By the hypotheses, $p$ must be induced by the unique Dynkin
diagram automorphism $\p$ of order three that sends node 0 to 1.

To verify that this affine structure satisfies Definition~\ref{d:cas}, we need
to check that restricting $B$ to $\{0, i\}$-arrows is a crystal for all $i \in
I$.  Each of these restrictions corresponds to a rank 2 classical crystal, and
Stembridge has given local rules in \cite{stembridge:2003} that characterize
such classical crystals in simply laced types.  Moreover, these rules depend
only on calculations involving $\varphi_i(b)$ and $\varepsilon_i(b)$ at each node $b\in B$.
Therefore, to check the restrictions for $i = 1, 2, 3, 4, 5$, it suffices by
Equation~(\ref{e:twist}) to apply $p$ and note that Stembridge's rules are
satisfied for the restriction of $B$ to $\{1, \p(i)\}$-arrows, since $B$ is a
classical crystal.  Here, $\p(i) = 6, 3, 5, 4, 2$, respectively.  
To check the restriction for $i = 6$, we use Equation~(\ref{e:as}) obtaining
\[ p p f_6 = p p f_6 p^{-1} p = p p^{-1} f_1 p p = f_1 p p \]
and
\[ p p f_0 = p p p f_6 p^{-1} = f_6 p p. \]
These imply that we can apply $p^2 = p^{-1}$ and note that Stembridge's rules
are satisfied for the restriction of $B$ to $\{6, 1\}$-arrows, since $B$ is
a classical crystal.  

Hence, we obtain a combinatorial affine structure for $B$.
\end{proof}

From now on, we use the notation $p$ to denote a twisted isomorphism induced by
$\p$ sending 
\[ 0 \mapsto 1 \mapsto 6 \mapsto 0, 2 \mapsto 3 \mapsto 5 \mapsto 2, 4 \mapsto 4. \]
Also, we let $\p$ act on the affine weight lattice as in
Remark~\ref{r:consequences}(2).


\bigskip
\subsection{Composition graphs}
\label{ss:CG}

Let $I = \{1, 2, \ldots, 6\}$ be the index set for the Dynkin diagram of $E_6$,
and $\widetilde{I} = I \cup \{0\}$ be the index set of $E_6^{(1)}$.
Suppose $J \subset I$.  Consider a
classical crystal $B$ of the form $\bigoplus B(k \Lambda)$ where $\Lambda$ is a
fundamental weight and we sum over some collection of nonnegative integers $k$.
Let $H^{J}(B)$ denote the $(I \setminus J)$-highest weight nodes of $B$.  By
incorporating the level 0 hypothesis of Remark~\ref{r:reality}, we also consider
the $(\widetilde{I} \setminus J)$-highest weight nodes of $B$ denoted by
$H^{J;0}(B)$.  

Our general strategy to define a twisted isomorphism $p$ on a classical
crystal $B$ will be to first define $p$ on $H^J(B)$, and then extend this
definition to the rest of $B$ using Equation~(\ref{e:twist}).  To accomplish
this, we introduce the following model for the nodes in $H^J(B)$ and
$H^{J;0}(B)$.

\begin{definition}\label{d:composition_graph}
Fix $J \subset I$ and form directed graphs $G_J$ and $G_{J;0}$ as follows.  

We construct the vertices of $G_J$ and $G_{J;0}$ iteratively, beginning with
all of the $(I \setminus J)$-highest weight nodes of $B(\Lambda)$.
Then, we add all of the vertices $b \in B(\Lambda)$ such that
\[ \{i \in I : \varepsilon_i(b) > 0\} \subset J \union \{i \in I : \parbox[t]{4in}{ there
exists $b' \in G_J$ with $b \otimes b'$ pairwise weakly increasing and
$\varphi_i(b') > 0$ \} } \]
to $G_J$.  Moreover, if $b$ also satisfies the property that there exists $b'
\in G_{J;0}$ with $b \otimes b'$ pairwise weakly increasing and $\wt_0(b') > 0$
whenever $\wt_0(b) < 0$, then we add $b$ to $G_{J;0}$.  We repeat this
construction until no new vertices are added.  This process eventually
terminates since $B(\Lambda)$ is finite.

The edges of $G_J$ and $G_{J;0}$ are determined by the pairwise weakly
increasing condition described in Definition~\ref{d:pwi}.  Note that some nodes
may have loops.  We call $G_J$ and $G_{J;0}$ the \em complete composition graph
\em for $J$ and $J;0$, respectively.
\end{definition}

\begin{lemma}\label{l:cg}
Every element of $H^{J}(B)$ and $H^{J;0}(B)$ is a pairwise weakly increasing
tensor product of vertices that form a directed path in $G_J$, respectively
$G_{J;0}$, where the element in $B(0)\subset H^J(B)$ is identified with the
empty tensor product.
\end{lemma}
\begin{proof}
We induct on the number of tensor factors $k$ to show that the algorithm in
Definition~\ref{d:composition_graph} produces all of the elements of $H^J(B)$
and $H^{J;0}(B)$ from component $B(k \Lambda)$.  The base case of $k = 1$ is
satisfied because we initially add all of the $(I \setminus J)$-highest weight
nodes of $B(\Lambda)$ to the complete composition graph, and $H^{J;0}(B) \subset
H^{J}(B)$.

For the induction step, observe that we branch on the left by the tensor product
rule.  That is, when $b \otimes b'$ is highest weight, we must have that $b'$ is
highest weight.  If there exists $b \in G_J$ with $\varepsilon_i(b) > 0$ where
\[ i \notin J \union \{i \in I : \varphi_i(b') > 0 \text{ for some } b' \in G_J
\text{ such that $b \otimes b'$ is pairwise weakly increasing } \} \]
then no tensor product of nodes that includes $b$ can ever cancel the $+$ from
$\varepsilon_i(b) > 0$ in the tensor product rule.  Therefore, no tensor
product of nodes that includes $b$ can be $(I \setminus J)$-highest weight.

Similarly for the case of $J;0$, if there exists $b \in G_{J;0}$ with $\wt_0(b) <
0$ and there is no $b' \in G_{J;0}$ with $b \otimes b'$ pairwise weakly increasing
and $\wt_0(b') > 0$, then we can conclude that $\wt_0$ of any tensor product of
nodes that starts with $b$ is negative.  Since every rightmost factor of a
highest weight tensor product must be highest weight, this would imply that no
tensor product of nodes that includes $b$ can be $(\widetilde{I} \setminus
J)$-highest weight.  

Hence, every $(I \setminus J)$-highest weight node is given by a pairwise
weakly increasing tensor product of vertices from $G_J$, and every
$(\widetilde{I} \setminus J)$-highest weight node is given by a pairwise weakly
increasing tensor product of vertices from $G_{J;0}$.
\end{proof}

We say that the vertices of $G_J$ are \em transitively closed \em if $b \otimes
c$ is pairwise weakly increasing whenever $b \otimes b'$ and $b' \otimes c$ are
pairwise weakly increasing for all $b, b', c \in G_J$.  Although it is not
obvious from Definition~\ref{d:pwi} whether the pairwise weakly increasing
condition is generally transitive, it is always a finite computation to verify
that the vertices of $G_J$ are transitively closed when $J$ is fixed.  Moreover,
it is straightforward to verify that all of the vertex sets that explicitly
appear in this work are transitively closed.

Therefore, we will typically draw only those edges of the complete composition
graph $G_J$ that cannot be inferred by transitivity, and we refer to this as the
\em (reduced) composition graph\em.  We will also abuse notation and refer to
this reduced composition graph as $G_J$.  We say that a \em chain \em is any
collection of vertices that form a subgraph of a directed path in a reduced
composition graph.  Lemma~\ref{l:cg} shows that we may identify nodes of
$H^{J}(B)$ from component $B(k \Lambda)$ of $B$ with chains in the reduced
composition graph $G_J$ having exactly $k$ vertices.  Analogues of all the
definitions and statements given in the previous two paragraphs hold for
$G_{J;0}$ and $H^{J;0}(B)$ as well.

We will see several examples of composition graphs in the following sections.


\bigskip
\subsection{Affine structures associated to $\Lambda_1$ and $\Lambda_6$}
\label{ss:L1}

Let $r \in \{1, 6\}$.  By~\cite[Proposition 3.4.4]{KKMMNN:1992}, a crystal basis
for the Kirillov--Reshetikhin module associated to $s \Lambda_r$ exists.  We
denote this crystal by $B^{r,s}$.  It follows from \cite{Chari:2001} that
$B^{r,s} \cong B(s \Lambda_r)$ as classical crystals.  In this section, we
construct a combinatorial model for $B^{r,s}$ in the sense of 
Definition~\ref{d:cas} using the order 3 Dynkin diagram automorphism of
$E_6^{(1)}$. 

Let $I=\{1,2,3,4,5,6\}$ be the index set of the $E_6$ Dynkin diagram, $J =
\{0, 2, 3, \ldots, 6\}$, and $K=I\setminus \{1\}=\{2,3,4,5,6\}$.  In this section, we use the weakly increasing
characterization given in Proposition~\ref{p:L1}.  This characterization implies
that the pairwise weakly increasing condition is transitive, so we draw reduced
composition graphs.

\bigskip

\begin{figure}[ht]
\begin{tabular}{c}
\xymatrix{ 
\bar{0}1 \ar[r] \ar@(ul,ur) & \bar{0}\bar{1}3 \ar[r] \ar@(ul,ur) & \bar{1}6 \ar@(ul,ur) \\
}
\end{tabular}
\caption{Composition graph for $I \setminus \{1\}$-highest weight nodes in $B(\Lambda_1)$}
\label{f:L1_01_hw_graph}
\end{figure}
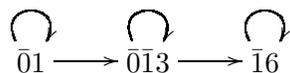

\begin{lemma}\label{lemma:mult free}
For $r\in \{1,6\}$, the $K$-highest weight nodes in
$B(s\Lambda_r)$ are distinguished by their $K$-weights.
\end{lemma}
\begin{proof}
The composition graph for the $K$-highest weight nodes for $B(\Lambda_1)$ is shown in
Figure~\ref{f:L1_01_hw_graph}.  Therefore, by Lemma~\ref{l:cg} all of the $K$-highest weight 
nodes in $B(s\Lambda_1)$ are of the form 
\[ \bar{0}1^{\otimes a} \otimes \bar{0}\bar{1}3^{\otimes b} \otimes \bar{1}6^{\otimes c} \]
and these nodes are all distinguished by their \{3, 6\}-weight together with
$s = a+b+c$.

Similarly, the $K$-highest weight nodes for $B(s\Lambda_6)$ are of the form
\begin{equation*}
\bar{0}6^a \otimes \bar{0} \bar{1}2^b \otimes \bar{1}0^c
\end{equation*}
which are also distinguished by their $K$-weight for fixed $s=a+b+c$.
\end{proof}

\bigskip

\begin{figure}[ht]
\begin{tabular}{c}
\xymatrix{ 
\bar{0}1 \ar[r] \ar@(ul,ur) & \bar{0}\bar{6}2 \ar[r] \ar@(ul,ur) & \bar{6}0 \ar@(ul,ur) \\
}
\end{tabular}
\caption{Composition graph for $I \setminus \{6\}$-highest weight nodes in $B(\Lambda_1)$}
\label{f:L1_06_hw_graph}
\end{figure}
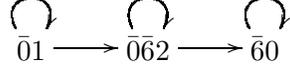

\begin{theorem} \label{t:E6_16}
Let $r\in \{1,6\}$ and $s\ge 1$.  There exists a unique twisted isomorphism $p : B(s\Lambda_r)
\to B(s\Lambda_r)$ of order three, such that node $b\in B(s\Lambda_r)$ is mapped to
node $p(b)$ with affine level-0 weight $\p(\wt(b))$.  
\end{theorem}
\begin{proof}
We state the proof for $r=1$. The proof for $r=6$ is analogous.

By constructing the composition graph shown in Figure~\ref{f:L1_06_hw_graph} and
applying Lemma~\ref{l:cg}, the $I \setminus \{6\}$-highest weight nodes of $B(s
\Lambda_1)$ all have the form 
\[ \bar{0}1^{\otimes a} \otimes \bar{0}\bar{6}2^{\otimes b} \otimes \bar{6}0^{\otimes c}. \]
All of these nodes are uniquely determined by their affine level-0 weight
\[ (c-b-a) \Lambda_0 + a \Lambda_1 + b \Lambda_2 - (b+c) \Lambda_6. \]
Any twisted isomorphism $p$ induced by $\p$ must send such a node to one which is $I
\setminus \{1\}$-highest weight, with affine level-0 weight 
\[ (c-b-a) \Lambda_1 + a \Lambda_6 + b \Lambda_3 - (b+c) \Lambda_0. \]
As we have seen in the proof of Lemma~\ref{lemma:mult free}, the $I \setminus
\{1\}$-highest weight nodes all have the form
\[ \bar{0}1^{\otimes a'} \otimes \bar{0}\bar{1}3^{\otimes b'} \otimes \bar{1}6^{\otimes c'} \]
and are all uniquely determined by their affine level-0 weight
\[ -(a'+b') \Lambda_0 + (a'-b'-c') \Lambda_1 + b' \Lambda_3 + c' \Lambda_6. \]
This system has the unique solution
\[ a' = c, b' = b, c' = a, \]
and we can extend by Equation~(\ref{e:twist}) to define $p$ on all of $B(s
\Lambda_1)$. 

If we apply $p$ again, we send the $I \setminus \{1\}$-highest weight nodes to
$\widetilde{I} \setminus \{6,1\}$-highest weight nodes with affine level-0
weight $-(a'+b') \Lambda_1 + (a'-b'-c') \Lambda_6 + b' \Lambda_5 + c'
\Lambda_0$.  This is accomplished by sending $\bar{0}1^{\otimes a'} \otimes
\bar{0}\bar{1}3^{\otimes b'} \otimes \bar{1}6^{\otimes c'}$ to
$\bar{1}6^{\otimes a'} \otimes \bar{1}\bar{6}5^{\otimes b'} \otimes
\bar{6}0^{\otimes c'}$.  Finally, observe that $p$ sends these $\widetilde{I}
\setminus \{6,1\}$-highest weight nodes to $I \setminus \{6\}$-highest weight
nodes with weight $-(a'+b') \Lambda_6 + (a'-b'-c') \Lambda_0 + b' \Lambda_2 + c'
\Lambda_1.$ Therefore, the twisted isomorphism $p$ has order three.
\end{proof}

\begin{corollary} \label{c:E6_16}
Let $s\ge 1$.  The twisted isomorphism $p$ of
Theorem~\ref{t:E6_16} defines a combinatorial affine crystal structure
$\widetilde{B(s\Lambda_1)}$ on $B(s\Lambda_1)$.  Moreover, if we restrict the
arrows in $\widetilde{B(s\Lambda_1)}$ to $J$, which we denote by
$\widetilde{B(s\Lambda_1)}|_J$, then 
\begin{equation} \label{e:Btilde}
 \widetilde{B(s\Lambda_1)}|_J \cong B(s \Lambda_6).
\end{equation}
\end{corollary}
The analogue of Corollary~\ref{c:E6_16} for $B(s\Lambda_6)$ also exists.
\begin{proof}
Since $p$ of Theorem~\ref{t:E6_16} has order three, it defines a combinatorial
affine structure on $B(s \Lambda_1)$ by Theorem~\ref{t:main}.  

Any $J$-highest weight node $b$ must also be an $I \setminus \{1\}$-highest
weight node, and these all have the form
\[ \bar{0}1^{\otimes a'} \otimes \bar{0}\bar{1}3^{\otimes b'} \otimes \bar{1}6^{\otimes c'} \]
with affine level-0 weight
\[ -(a'+b') \Lambda_0 + (a'-b'-c') \Lambda_1 + b' \Lambda_3 + c' \Lambda_6. \]
If we further require that $\wt_0(b) \geq 0$, then we see that $a'$ and $b'$ must
be 0.  Hence, $b = \bar{1}6^{\otimes s}$ with $J$-weight $s\Lambda_6$.
\end{proof}

\begin{theorem}\label{t:16unique}
Let $\widetilde{B}, \widetilde{B'}$ be two affine type $E_6^{(1)}$ crystals. Suppose there exists
a $\{1,2,3,4,5,6\}$-isomorphism $\Psi_0$ and a $\{0,2,3,4,5,6\}$-isomorphism $\Psi_1$ where
\begin{equation} \label{eq:isom}
\begin{split}
    \widetilde{B}|_{\{1,2,3,4,5,6\}} &\stackrel{\Psi_0}{\longrightarrow} \widetilde{B'}|_{\{1,2,3,4,5,6\}} \cong B(s\Lambda_1), \\ 
    \widetilde{B}|_{\{0,2,3,4,5,6\}} &\stackrel{\Psi_1}{\longrightarrow} \widetilde{B'}|_{\{0,2,3,4,5,6\}} \cong B(s\Lambda_6).
\end{split}
\end{equation}
Then $\Psi_0(b)=\Psi_1(b)$ for all $b\in \widetilde{B}$ and so there exists an
$\{0,1,2,3,4,5,6\}$-isomorphism $\Psi: \widetilde{B} \cong \widetilde{B'}$.
\end{theorem}
\begin{proof}
Set $K=\{2,3,4,5,6\}$. Note that if $\Psi_0(b)=\Psi_1(b)$ for a $b$ in a given
$K$-component $\mathcal{C}$, then $\Psi_0(b')=\Psi_1(b')$ for all $b'\in
\mathcal{C}$ since $e_i\Psi_0(b')=\Psi_0(e_ib')$ and
$e_i\Psi_1(b')=\Psi_1(e_ib')$ for $i\in K$.  Furthermore, observe that $\Psi_0$
and $\Psi_1$ preserve weights by Remark~\ref{r:reality}.  That is,
$\wt(b)=\wt(\Psi_0(b))=\wt(\Psi_1(b))$ for all $b\in \widetilde{B}$.

Since $e_i$ commutes with $\Psi_0$ and $\Psi_1$ for $i\in K$, it follows that
$K$-components in $\widetilde{B}$ must map to $K$-components in
$\widetilde{B'}$.  Restricted to $I$ or $J$, the images of the $K$-components
in $\widetilde{B}$ are also isomorphic to $K$-components in $B(s\Lambda_1)$
under $\Psi_0$ and to $K$-components in $B(s\Lambda_6)$ under $\Psi_1$.
However, the $K$-highest weight elements in $B(s\Lambda_1)$ and $B(s\Lambda_6)$
are determined by their weights by Lemma~\ref{lemma:mult free}. Hence we must
have $\Psi_0(b)=\Psi_1(b)$ for all $b\in \widetilde{B}$.
\end{proof}

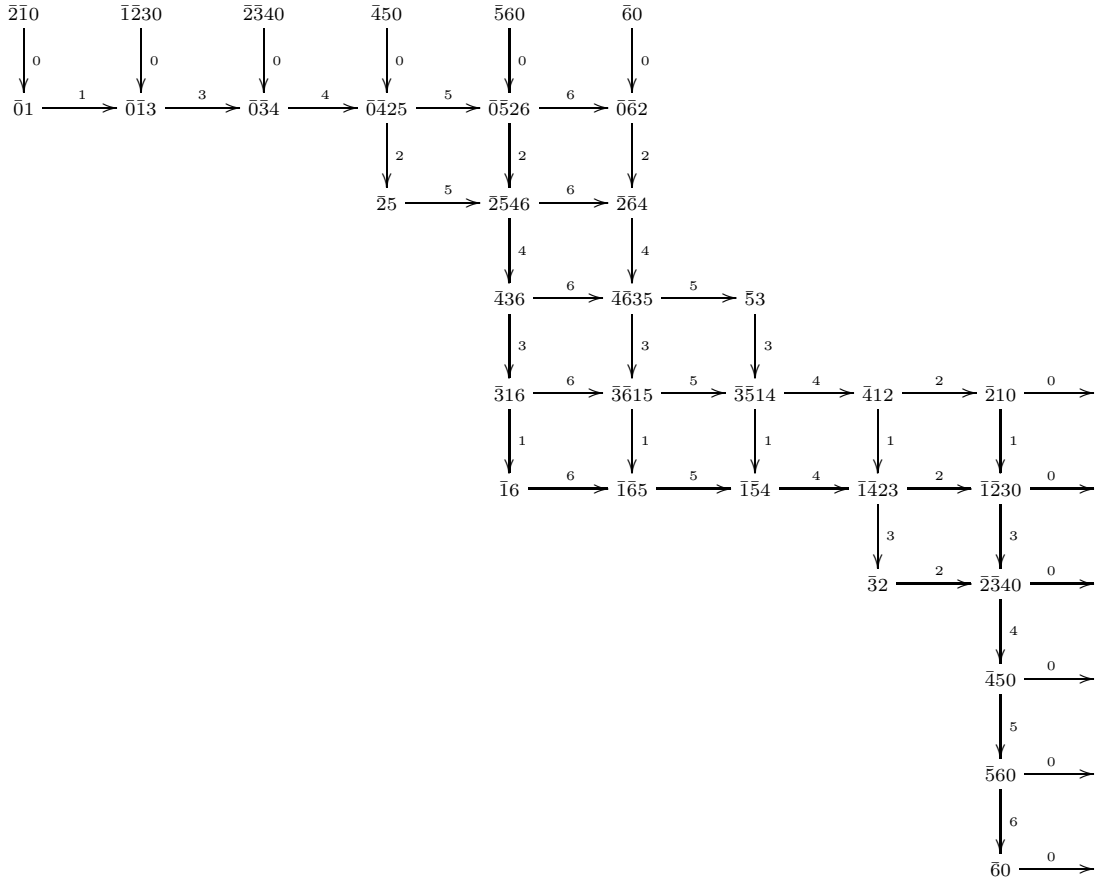
\begin{figure}[ht]
\begin{center}
\tiny{ \xymatrix{
\bar{2}\bar{1}0 \ar[d]^0 & \bar{1}\bar{2}30 \ar[d]^0 & \bar{2}\bar{3}40 \ar[d]^0 & \bar{4}50 \ar[d]^0  & \bar{5}60 \ar[d]^0 & \bar{6}0 \ar[d]^0 \\
\bar{0}1 \ar[r]^1 & \bar{0}\bar{1}3 \ar[r]^3 & \bar{0}\bar{3}4 \ar[r]^4 & \bar{0}\bar{4}25 \ar[r]^5 \ar[d]^2 & \bar{0}\bar{5}26 \ar[r]^6 \ar[d]^2 & \bar{0}\bar{6}2 \ar[d]^2 \\
                  &                          &                          & \bar{2}5 \ar[r]^5                  & \bar{2}\bar{5}46 \ar[r]^6 \ar[d]^4 & \bar{2}\bar{6}4 \ar[d]^4 \\
                  &                          &                          &                                    & \bar{4}36 \ar[r]^6 \ar[d]^3        & \bar{4}\bar{6}35 \ar[r]^5 \ar[d]^3 & \bar{5}3 \ar[d]^3  \\
                  &                          &                          &                                    & \bar{3}16 \ar[r]^6 \ar[d]^1        & \bar{3}\bar{6}15 \ar[r]^5 \ar[d]^1 & \bar{3}\bar{5}14 \ar[r]^4 \ar[d]^1 & \bar{4}12 \ar[r]^2 \ar[d]^1 & \bar{2}10 \ar[d]^1 \ar[r]^0 &  \\
                  &                          &                          &                                    & \bar{1}6 \ar[r]^6                  & \bar{1}\bar{6}5 \ar[r]^5           & \bar{1}\bar{5}4 \ar[r]^4           & \bar{1}\bar{4}23 \ar[r]^2 \ar[d]^3 & \bar{1}\bar{2}30 \ar[d]^3 \ar[r]^0 & \\
                  &                          &                          &                                    &                                    &                                    &                                    & \bar{3}2 \ar[r]^2           & \bar{2}\bar{3}40 \ar[d]^4 \ar[r]^0 & \\
                  &                          &                          &                                    &                                    &                                    &                                    &                             & \bar{4}50 \ar[d]^5 \ar[r]^0 & \\
                  &                          &                          &                                    &                                    &                                    &                                    &                             & \bar{5}60 \ar[d]^6 \ar[r]^0 & \\
                  &                          &                          &                                    &                                    &                                    &                                    &                             & \bar{6}0  \ar[r]^0 & \\
} }
\end{center}
\caption{Crystal graph for $B^{1,1}$ of type $E_6^{(1)}$}\label{f:aBL1}
\end{figure}

\begin{corollary}
For $r \in \{1, 6\}$ and $s\ge 1$, the combinatorial affine structure $\widetilde{B(s\Lambda_r)}$
of Corollary~\ref{c:E6_16} is isomorphic to the Kirillov--Reshetikhin crystal $B^{r,s}$.
\end{corollary}
\begin{proof}
By~\cite{Chari:2001}, $B^{r,s}\cong B(s\Lambda_r)$ for $r=1,6$ as a classical crystal.
By~\cite[Proposition 3.4.4]{KKMMNN:1992}, $B^{r,s}$ for $r=1,6$ exists since it is irreducible 
as a classical crystal.

Let us now restrict to $r=1$ as the case $r=6$ is analogous.
To show that $B^{1,s} \cong B(s\Lambda_6)$ as a $J$-crystal, 
it suffices to show that there exists a corresponding highest weight vector since the crystal is 
irreducible. However, the element of level-0 weight $s(\Lambda_6-\Lambda_1)$
is precisely this element.

Since $B^{1,r}|_I \cong \widetilde{B(s\Lambda_1)}|_I \cong B(s\Lambda_1)$
and  $B^{1,r}|_J \cong \widetilde{B(s\Lambda_1)}|_J \cong B(s\Lambda_6)$
by the above arguments and~\eqref{e:Btilde}, by Theorem~\ref{t:16unique}
we must have $B^{1,s} \cong \widetilde{B(s\Lambda_1)}$ as affine crystals.
\end{proof}

The resulting affine crystal $B^{1,1}$ is shown in Figure~\ref{f:aBL1}.


\bigskip
\subsection{Affine structures associated to $\Lambda_2$}
\label{ss:L2}

By \cite[Proposition 3.4.5]{KKMMNN:1992}, a crystal basis $B^{2,s}$ for the
Kirillov--Reshetikhin module associated to $s \Lambda_2$ exists.
 It follows from \cite{Chari:2001} that $B^{2,s} \cong
\bigoplus_{k = 0}^s B(k \Lambda_2)$ as classical crystals.  We will refer to
$B(k \Lambda_2)$ as the \em $k$th component \em of $\bigoplus_{k = 0}^s B(k
\Lambda_2)$.  In this section, we will show how to construct a combinatorial
affine structure for $\bigoplus_{k = 0}^s B(k \Lambda_2)$ using
Theorem~\ref{t:main}.  

We use the weakly increasing characterization given in Proposition~\ref{p:L2wi}
for our work in this section.  Let $H_s^{J}$ denote the $(I \setminus
J)$-highest weight nodes of $\bigoplus_{k = 0}^s B(k \Lambda_2)$.  The
composition graphs for $J = \{6\}$ and $J = \{1\}$ are shown in
Figures~\ref{f:06_hw_graph} and \ref{f:10_hw_graph}, respectively.  Observe that
the nodes $a$ and $c$ were added to $H_1^{\{6\}}$ in the course of the algorithm
described in Definition~\ref{d:composition_graph} to obtain $G_{6}$.  The nodes
of weight 0 do not have loops by Proposition~\ref{p:L2wi}.  A finite computation
shows that the vertex sets of these composition graphs are transitively closed,
so Lemma~\ref{l:cg} models the nodes of $H_s^{\{6\}}$ and $H_s^{\{1\}}$ as
chains in $G_{6}$ and $G_{1}$, respectively.

\bigskip

\begin{figure}[ht]
\begin{tabular}{c}
\xymatrix{ 
\stackrel{2\bar{1}\bar{0} \otimes \bar{0}1}{u} \ar[r] \ar@(ul,ur) & \stackrel{36\bar{1}\bar{5} \otimes \bar{0}1}{a} \ar[r] \ar[d] \ar@(ul,ur) & \stackrel{3\bar{1}\bar{6} \otimes \bar{0}1}{b} \ar[dr] \ar@(ul,ur) & \\
& \stackrel{06\bar{2} \otimes \bar{0}\bar{5}26}{c} \ar[r] \ar@(dl,dr) & \stackrel{06\bar{2} \otimes \bar{0}\bar{6}2}{d} \ar[r] & \stackrel{05\bar{2}\bar{6} \otimes \bar{0}\bar{6}2}{e} \ar@(ul,ur) \\
}
\end{tabular}
\caption{Composition graph $G_{6}$ for $I \setminus \{6\}$-highest weight nodes}\label{f:06_hw_graph}
\end{figure}
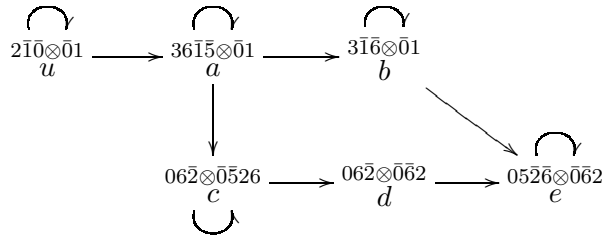

\begin{figure}[ht]
\begin{tabular}{c}
\xymatrix{ 
\stackrel{02\bar{1} \otimes \bar{0}1}{u'} \ar[r] \ar@(ul,ur) & \stackrel{5\bar{3} \otimes \bar{0}1}{a'} \ar[r] \ar[d]  \ar@(ul,ur) & \stackrel{5\bar{3} \otimes \bar{0}\bar{1}3}{b'} \ar[dr] \ar@(ul,ur) & \\
& \stackrel{01\bar{3} \otimes \bar{0}1}{c'} \ar[r] \ar@(dl,dr) & \stackrel{01\bar{3} \otimes \bar{0}\bar{1}3}{d'} \ar[r] & \stackrel{0\bar{1} \otimes \bar{0}\bar{1}3}{e'}  \ar@(ul,ur) \\
}
\end{tabular}
\caption{Composition graph $G_{1}$ for $I \setminus \{1\}$-highest weight nodes}\label{f:10_hw_graph}
\end{figure}

\bigskip

\begin{example}
We see from the composition graph that 
\[ (2\bar{1}\bar{0} \otimes \bar{0}1) \otimes (2\bar{1}\bar{0} \otimes \bar{0}1)
\otimes (06\bar{2} \otimes \bar{0}\bar{5}26) \otimes (05\bar{2}\bar{6} \otimes
\bar{0}\bar{6}2) \]
is a typical node in $H_4^{\{6\}}$.
\end{example}

\begin{definition}
Let $\mathcal{C}(m)$ denote the set
\[ \{ (L_2, L_3, L_5) \in \mathbb{Z}_{\geq 0} : L_2 + L_3 + L_5 = m \} \]
of \em weak compositions \em of $m$ into 3 parts.
\end{definition}

\begin{proposition}\label{p:weak_comp}
There is a bijection from the $I \setminus \{6\}$-highest weight nodes of
$B(k \Lambda_2)$ to $\bigcup_{m=0}^{k} \mathcal{C}(m)$ such that a node corresponding to
the weak composition $L_2 + L_3 + L_5 = m$ has $I \setminus \{6\}$-weight $L_2
\Lambda_2 + L_3 \Lambda_3 + L_5 \Lambda_5$.

In particular, the $I \setminus \{6\}$-highest weight nodes of $B(k
\Lambda_2)$ are determined by their $\{2,3,5\}$-weight, and for any such node
$b$, we have
\[ k =  \varphi_6(b) + \wt_2(b) + \wt_3(b) + \wt_5(b). \]
\end{proposition}
\begin{proof}
By Lemma~\ref{l:cg}, the $I \setminus \{6\}$-highest weight nodes of $B(k \Lambda_2)$ 
correspond to chains of length $k$ in $G_{6}$.  Moreover, we claim
that for each value of $k$ and weak composition $L_2 + L_3 + L_5 = m$ with $0
\leq m \leq k$, there exists a unique chain of length $k$ in $G_{6}$ having $I
\setminus \{6\}$-weight $L_2 \Lambda_2 + L_3 \Lambda_3 + L_5 \Lambda_5$.

Denote the multiplicities of the vertices by $u, a, b, c, d, e$ corresponding to
the labeling in Figure~\ref{f:06_hw_graph}.  All of these multiplicities must be
nonnegative, and we also have $d \in \{0, 1\}$ by Proposition~\ref{p:L2wi}.
There are two maximal chains in $G_{6}$ and we will write a system of linear
equations for each of them.

The equations among the multiplicities that are induced by the upper maximal
chain of the graph are
\begin{align*}
L_2 &= u      && L_5 = e-a\\
L_3 &= a+b && k = u + a + b + e
\end{align*}
and we can solve these to obtain
\begin{equation*}
\begin{split}
a &= k - (L_2 + L_3 + L_5)\\
e &= k - (L_2 + L_3)\\
b &= 2 L_3 + L_5 + L_2 - k.
\end{split}
\end{equation*}
Note that $a, e \geq 0$, but $b$ may be $<0$.

The equations induced by the lower maximal chain are
\begin{align*}
L_2 &= u   && L_5 = e-c-a\\
L_3 &= a  && k= u+a+c+d+e
\end{align*}
and we can solve these to obtain
\[ 2e + d = k + L_5 - L_2 \]
which has a unique solution in nonnegative integers with $d \in \{1,0\}$, and
\[ c = e - (L_5 + L_3). \]
Now, $a, d, e \geq 0$. But $c \geq 0$ if and only if $2c \geq 0$ if and only if
\[ k + L_5 - L_2 - d - 2 (L_5 + L_3) = k - L_2 - L_5 - 2 L_3 - d \geq 0. \]
This occurs when $d = 0$ and $b \leq 0$ or when $d = 1$ and $b < 0$.  Moreover,
the solutions for the two chains in the graph agree when $b = 2 L_3 + L_5 + L_2 -
k = 0$.  
Hence, we obtain a unique solution in all cases of the parameters $k,
L_2, L_3, L_5$.

In addition, we have that $\varphi_6$ and $\varepsilon_6$ are uniquely
determined by $L_2, L_3, L_5$ and $k$.  The upper path equations give
\[ \varphi_6 = a = k - L_2 - L_3 - L_5 \text{ and } \varepsilon_6 = b+2e = k - L_2 + L_5. \]
The lower path equations give
\[ \varphi_6 = a+2c+d = L_3 + k + L_5 - L_2 - 2(L_3 + L_5) = k - L_2 - L_3 -
L_5 \text{ and } \varepsilon_6 = d+2e = k - L_2 + L_5. \]
So $\varphi_6$ and $\varepsilon_6$ agree in both cases.

Finally, $\varepsilon_i$ and $\varphi_i$ for $i=1,4$ of any solution is zero.
\end{proof}

\begin{remark}
Proposition~\ref{p:weak_comp} can also be interpreted as a branching rule from classical 
$E_6$ to $D_5$.  
\end{remark}

\begin{corollary}\label{c:L2hw}
The $I \setminus \{6\}$-highest weight nodes of $\bigoplus_{k = 0}^s B(k
\Lambda_2)$ are uniquely determined by their $\{2,3,5\}$-weight together with
$\varphi_6$.
The $I \setminus \{1\}$-highest weight nodes of $\bigoplus_{k = 0}^s B(k
\Lambda_2)$ are uniquely determined by their $\{2,3,5\}$-weight together with
$\varphi_1$.
\end{corollary}
\begin{proof}
The first statement follows directly from Proposition~\ref{p:weak_comp}, and the
second statement has an analogous proof.

The composition graph for the $I \setminus \{1\}$-highest weight nodes is
shown in Figure~\ref{f:10_hw_graph}.  Note that $d' \otimes d'$ is not weakly
increasing.  When we set up the analogous set of equations to solve for the
multiplicities $u', a', b', c', d', e'$ in terms of the parameters $k, L_2,
L_3, L_5$, we obtain equations derived from those in the proof of
Proposition~\ref{p:weak_comp} by fixing $\Lambda_2$ and interchanging
$\Lambda_3$ with $\Lambda_5$.
\end{proof}

We are now in a position to state our main result.  

\begin{theorem}\label{t:L2_main}
There exists a unique twisted isomorphism $p: \bigoplus_{k = 0}^s B(k \Lambda_2)
\rightarrow \bigoplus_{k = 0}^s B(k \Lambda_2)$ of order three.  This isomorphism
sends an $I \setminus \{6\}$-highest weight node $b$ from component $k$ to the
unique $I \setminus \{1\}$-highest weight node $b'$ in component $(s-k) +
(\wt_2(b) + \wt_3(b) + \wt_5(b))$ satisfying $\wt_{\p(i)}(b') = \wt_i(b)$ for
each $i \in \{2, 3, 5\}$.
\end{theorem}
The proof of this theorem is given at the end of this section. We first discuss
some consequences, examples, and preliminary results.

\begin{corollary} \label{c:L2_main}
The twisted isomorphism $p$ of Theorem~\ref{t:L2_main} defines a combinatorial
affine crystal structure which is isomorphic to the Kirillov--Reshetikhin crystal $B^{2,s}$.
\end{corollary}
\begin{proof}
By Theorem~\ref{t:main}, $p$ yields a combinatorial affine structure for 
$\bigoplus_{k = 0}^s B(k\Lambda_2)$ via Equation~\eqref{e:twist}.
The results of Chari~\cite{Chari:2001} show that $B^{2,s}$ has the same 
classical decomposition.
By \cite[Theorem 6.1]{KMOY:2007}, we have that if a combinatorial affine
structure for $\bigoplus_{k = 0}^s B(k \Lambda_2)$ exists, then it is isomorphic to 
the Kirillov--Reshetikhin crystal $B^{2,s}$.  
\end{proof}

\begin{example}
Suppose $s = 3$.  Then, $H_s^{\{6\}}$ decomposes into $(s+1)$ components according
to which summand $B(k \Lambda_2)$ the node lies in.  Each of these components
further decomposes as $\bigcup_{m=0}^k \mathcal{C}(m)$ by
Proposition~\ref{p:weak_comp}.  Hence, we have the following schematic of
$H_s^{\{6\}}$ in which the twisted isomorphism $p$ reflects the $\mathcal{C}(m)$
components along rows.  The twisted isomorphism $p$ also twists the weights
according to $\p$, which is not shown explicitly.  The resulting node lies in
$H_s^{\{1\}}$.
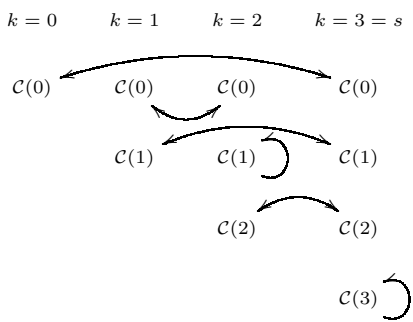
\[
{ \tiny 
\xymatrix@-0.9pc{
k=0 & k=1 & k=2 & k=3=s \\
\mathcal{C}(0) \ar@/^1pc/@{<->}[rrr] & \mathcal{C}(0) \ar@/_1pc/@{<->}[r]&
\mathcal{C}(0) & \mathcal{C}(0) \\
& \mathcal{C}(1) \ar@/^1pc/@{<->}[rr] & \mathcal{C}(1) \ar@(dr,ur) &
\mathcal{C}(1) \\
&  & \mathcal{C}(2) \ar@/^1pc/@{<->}[r] & \mathcal{C}(2) \\
&  &  & \mathcal{C}(3) \ar@(dr,ur) \\
   } } \]
To compute $p(b)$ for 
\[ b = (2\bar{1}\bar{0} \otimes \bar{0}1) \otimes (06\bar{2} \otimes \bar{0}\bar{5}26) \otimes (05\bar{2}\bar{6} \otimes \bar{0}\bar{6}2) \]
we observe that $\wt_2(b) = 1$, $\wt_3(b) = 0$, $\wt_5(b) = 0$ so the
composition associated $b$ is $(1,0,0)$.  According to Theorem~\ref{t:L2_main},
$p$ maps $b$ to the unique chain of length 1 in $G_{1}$ corresponding to the
composition $(0,1,0)$, namely $b' = 0\bar{1} \otimes \bar{0}\bar{1}3$.
In general, we define $f_0(b)$ by $p^{-1} \circ f_1 \circ p(b)$.  In this case,
$f_0(b) = 0$.
\end{example}

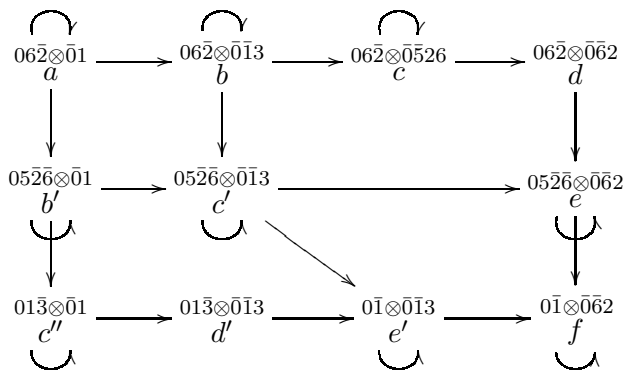
\begin{figure}[ht]
\begin{tabular}{c}
\xymatrix{ 
\stackrel{06\bar{2} \otimes \bar{0}1}{a} \ar[r] \ar[d] \ar@(ul,ur)  & \stackrel{06\bar{2} \otimes \bar{0}\bar{1}3}{b} \ar[r] \ar[d] \ar@(ul,ur) & \stackrel{06\bar{2} \otimes \bar{0}\bar{5}26}{c} \ar[r] \ar@(ul,ur) & \stackrel{06\bar{2} \otimes \bar{0}\bar{6}2}{d} \ar[d] \\ 
\stackrel{05\bar{2}\bar{6} \otimes \bar{0}1}{b'} \ar[r] \ar[d] \ar@(dl,dr) & \stackrel{05\bar{2}\bar{6} \otimes \bar{0}\bar{1}3}{c'} \ar[rr] \ar[dr] \ar@(dl,dr) &  & \stackrel{05\bar{2}\bar{6} \otimes \bar{0}\bar{6}2}{e} \ar[d] \ar@(dl,dr) \\ 
\stackrel{01\bar{3} \otimes \bar{0}1}{c''} \ar[r] \ar@(dl,dr) & \stackrel{01\bar{3} \otimes \bar{0}\bar{1}3}{d'} \ar[r] & \stackrel{0\bar{1} \otimes \bar{0}\bar{1}3}{e'} \ar[r] \ar@(dl,dr)  & \stackrel{0\bar{1} \otimes \bar{0}\bar{6}2}{f}  \ar@(dl,dr) \\ 
}
\end{tabular}
\caption{Graph $G_{6,1;0}$ of weakly increasing $\widetilde{I} \setminus \{6,1\}$-highest weight nodes}\label{f:61_hw_graph}
\end{figure}

The composition graph for the $(\widetilde{I} \setminus \{6,1\})$-highest weight
nodes is shown in Figure~\ref{f:61_hw_graph}.  This graph was constructed using
the algorithm described in Definition~\ref{d:composition_graph}.  It is more
complicated than the composition graphs $G_{6}$ and $G_{1}$ because we are
taking highest weight nodes with respect to the complement of two classical
Dynkin diagram nodes.  Also, we use the level 0 hypothesis to compute affine
weights and our composition graph includes only those nodes that can contribute
to chains having 0-highest weight.  A finite computation shows that the vertex
set of $G_{6,1;0}$ is transitively closed, so the $\widetilde{I} \setminus
\{6,1\}$-highest weight nodes correspond to chains in $G_{6,1;0}$ by
Lemma~\ref{l:cg}.

In order to prove Theorem~\ref{t:L2_main}, we study how $p$ maps chains from
$G_{1}$ to chains in $G_{6,1;0}$.

\begin{lemma}\label{l:solve_equations}
Let $b$ be an $I \setminus \{1\}$-highest weight node of $B(j \Lambda_2)$
corresponding to the weak composition $(L_2, L_3, L_5)$.  Then, for every $j
\leq k \leq s$, there exists a unique $\widetilde{I} \setminus \{6,1\}$-highest
weight node $b'$ in $B(k \Lambda_2)$ such that $\wt_i(b') = \wt_{\p^{-1}(i)}(b)
= L_{\p^{-1}(i)}$ for $i \in \{2, 3, 5\}$.  Moreover, $\varphi_1(b') = k-j$.
\end{lemma}
\begin{proof}
In Appendix~\ref{s:appendix}, we solve the equations describing how to map an $I
\setminus \{1\}$-highest weight node from component $j$ to an $\widetilde{I} \setminus
\{6,1\}$-highest weight node of component $k$, using the equation for
$\varphi_1$ from Corollary~\ref{c:L2hw} which must become $\varphi_6$ in the
image.

As shown in Appendix~\ref{s:appendix}, there is one system of linear equations for 
each of the 6 maximal chains in $G_{6,1;0}$.  The set of parameters for which each 
case is valid is shown below.
{ \tiny 
\begin{eqnarray*}
\text{ Case 1 } & (k-j) + L_3 \leq \varphi_6  \\
\text{ Case 2 } & (k-j) \leq \varphi_6 \leq (k-j) + L_3 \leq \varphi_6 + L_5 \\
\text{ Case 3 } & \varphi_6 \leq (k-j) \leq (k-j) + L_3 \leq \varphi_6 + L_5 \\
\text{ Case 4 } & \varphi_6 \leq (k-j) \leq \varphi_6 + L_5 \leq (k-j) + L_3 \\
\text{ Case 5 } & (k-j) \leq \varphi_6 \leq \varphi_6 + L_5 \leq (k-j) + L_3 \\
\text{ Case 6 } & \varphi_6 + L_5 < (k-j) \\
\end{eqnarray*}
}
Observe that these cover all possible values of the parameters, because if we
are not in Case 1 nor Case 6, then we have the partial order of
parameters shown below.
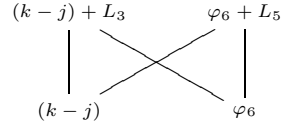
\[\tiny \xymatrix { (k-j)+L_3 \ar@{-}[d] \ar@{-}[dr] & \varphi_6 + L_5 \ar@{-}[d] \ar@{-}[dl] \\
(k-j) & \varphi_6 \\ } \]
This partial order has exactly four linear extensions corresponding precisely to
Cases 2-5.

We must also show that if a particular set of parameters $(L_2, L_3, L_5, j, k)$
is satisfied by multiple cases, then the solutions obtained from each case all
agree.  This can be done by hand for the systems described in
Appendix~\ref{s:appendix}.  In Appendix~\ref{s:appendix_b} we also describe an
effective procedure that can be automated to establish this fact.

Observe that in every solution, $k-j$ must be nonnegative.  Moreover, in every
case, $\varphi_1$ of the solution is $k-j$.  
\end{proof}

\begin{proof}[Proof of Theorem~\ref{t:L2_main}]
Fix a weight $L_2 \Lambda_2 + L_3 \Lambda_3 + L_5 \Lambda_5$ and a component $j
\leq s$.  There is a unique $I \setminus \{1\}$-highest weight node $b$
corresponding to these parameters by Corollary~\ref{c:L2hw}.  Any twisted
isomorphism $p$ induced from the Dynkin diagram automorphism $\p$ sends $b$ to
an $\widetilde{I} \setminus \{6,1\}$-highest weight node $p(b)$ in some
component, say $k$, and $p(p(b))$ is an $I \setminus \{6\}$-highest weight node
in some component, say $j'$.  

By Lemma~\ref{l:solve_equations}, we have that a solution $p(b)$ exists and that
$\varphi_1(p(b))$ is $(k-j)$. 
Hence,
\[ j' - (L_2 + L_3 + L_5) = (k - j) \geq 0 \]
by Corollary~\ref{c:L2hw}.

We suppose that $p$ has order three, and work by downward induction on $j$,
starting from the fact that nodes of component $j = s$ must go to component $k =
s$, which goes to component $j' = L_2 + L_3 + L_5$.  As $j$ decreases, if we
ever have $k < s$, then $\varphi_1$ with respect to $\widetilde{I} \setminus
\{6,1\}$ is less than $(s-j)$.  This implies that $j' < (s-j) + (L_2 + L_3 +
L_5)$, and so we would map $p(b)$ onto an $I \setminus \{6\}$-highest weight
node that has already appeared in the image of $p$.  Hence, we find that $k = s$
always.  This specifies a unique solution of order three for $p$.
\end{proof}


\bigskip
\subsection{A conjecture for $E_7$}
\label{ss:E7}

Recall the Dynkin diagram of type $E_7^{(1)}$ shown in Figure~\ref{fig:E67}.  Let
$\p$ denote the unique automorphism of this diagram, so $\p$ has order two and
sends the affine node 0 to node 7.

The adjoint node in $E_7$ is node 1, and \cite{Chari:2001} has given the
decomposition $B^{1, s} = \bigoplus_{k = 0}^{s} B(k \Lambda_1)$ of the
corresponding Kirillov--Reshetikhin crystal into classical crystals.  We can
form the composition graph for $J = \{7\}$ and the result is shown in
Figure~\ref{f:07_hw_graph}.

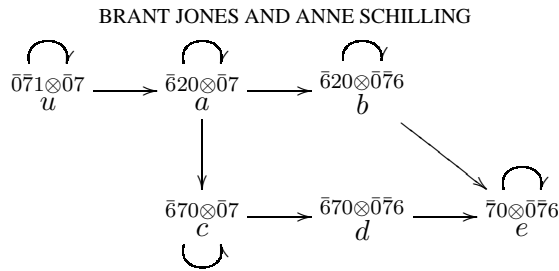
\begin{figure}[ht]
\begin{tabular}{c}
\xymatrix{ 
\stackrel{\bar{0}\bar{7}1 \otimes \bar{0}7}{u} \ar[r] \ar@(ul,ur) & \stackrel{\bar{6}20 \otimes \bar{0}7}{a} \ar[r] \ar[d]  \ar@(ul,ur) & \stackrel{\bar{6}20 \otimes \bar{0}\bar{7}6}{b} \ar[dr] \ar@(ul,ur) & \\
& \stackrel{\bar{6}70 \otimes \bar{0}7}{c} \ar[r] \ar@(dl,dr) & \stackrel{\bar{6}70 \otimes \bar{0}\bar{7}6}{d} \ar[r] & \stackrel{\bar{7}0 \otimes \bar{0}\bar{7}6}{e}  \ar@(ul,ur) \\
}
\end{tabular}
\caption{Composition graph $G_{7}$ for $I \setminus \{7\}$-highest weight nodes in 
$B(\Lambda_1)$ for $E_7$}\label{f:07_hw_graph}
\end{figure}

This graph is essentially the same as the composition graph $G_{1}$ that we
obtained for $B(\Lambda_2)$ in $E_6$.  In particular, the classical weights
$\Lambda_1, \Lambda_2, \Lambda_6, \Lambda_7$ that appear in $G_{7}$ for type $E_7$
correspond to $\Lambda_2, \Lambda_5, \Lambda_3, \Lambda_1$ in $G_{1}$ of type $E_6$. 
Our solution to the equations associated with $G_{1}$ in $E_6$ shows that there exists a unique 
$I \setminus \{7\}$-highest weight node of $B(k \Lambda_1)$ in $E_7$ having
weight $L_1 \Lambda_1 + L_2 \Lambda_2 + L_6 \Lambda_6$.  That is, the 
$I \setminus \{7\}$-highest weight nodes of $B(k \Lambda_1)$ are in bijection
with weak compositions with 3 parts.  Moreover, we have that $k = \varphi_7(b) +
\wt_1(b) + \wt_2(b) + \wt_6(b)$ for such nodes $b$.  

Define $p: \bigoplus_{k = 0}^{s} B(k \Lambda_1) \rightarrow \bigoplus_{k =
0}^{s} B(k \Lambda_1)$ on the $I \setminus \{7\}$-highest weight nodes by
sending $b \in B(k \Lambda_1)$ to the unique $I \setminus \{7\}$-highest
weight node $b'$ in component $(s-k) + (\wt_1(b) + \wt_2(b) + \wt_6(b))$
satisfying $\wt_{\p(i)}(b') = \wt_i(b)$ for each $i \in \{1, 2, 6\}$.

Since $\p$ does not have order three, Theorem~\ref{t:main} does not apply to prove
that this construction gives a combinatorial affine structure.  To get a sense
of the ambiguity that can arise when working with twisted isomorphisms of order
two, consider Example~\ref{e:prom2}.  It remains to show that if we define
0-arrows by $f_0 = p \circ f_7 \circ p$, then the restriction to $\{0,
i\}$-arrows is a crystal for all $i \in I$.  The argument given in the proof of
Theorem~\ref{t:main} shows that this is true for all $i \neq 7$.  Moreover, we
conjecture that this is true for $i = 7$ as well.

\begin{conjecture} \label{conj:E7}
Define $p: \bigoplus_{k = 0}^{s} B(k \Lambda_1) \rightarrow \bigoplus_{k =
0}^{s} B(k \Lambda_1)$ as described above, and let $f_0 = p \circ f_7 \circ p$.
Then $f_0$ commutes with $f_7$ so we obtain a combinatorial affine structure on
$\bigoplus_{k = 0}^{s} B(k \Lambda_1)$, which is isomorphic to $B^{1,s}$ of type
$E_7^{(1)}$.
\end{conjecture}


We have verified this conjecture for $s \leq 2$.


\bigskip
\section{\sage implementation}\label{s:sage_implementation}

As illustrated in the following examples, we have implemented the crystals described in this paper
in \sage~\cite{sage} and {\tt Sage-Combinat}~\cite{sage-combinat}.
For more information see the documentation of {\tt Sage-Combinat} and \sage,
in particular the crystal documentation
\footnote{\url{http://www.sagemath.org/doc/reference/combinat/crystals.html}}.

\begin{computerExample}
For type $E_6$, the building block $B(\Lambda_1)$ of Figure~\ref{fig:L1}
is accessible as follows:
\begin{verbatim}
sage: C = CrystalOfLetters(['E',6])
sage: C.list()
[[1], [-1, 3], [-3, 4], [-4, 2, 5], [-2, 5], [-5, 2, 6], [-2, -5, 4, 6],
[-4, 3, 6], [-3, 1, 6], [-1, 6], [-6, 2], [-2, -6, 4], [-4, -6, 3, 5],
[-3, -6, 1, 5], [-1, -6, 5], [-5, 3], [-3, -5, 1, 4], [-1, -5, 4], [-4, 1, 2],
[-1, -4, 2, 3], [-3, 2], [-2, -3, 4], [-4, 5], [-5, 6], [-6], [-2, 1], [-1, -2, 3]]
\end{verbatim}
The crystal can be plotted as 
\begin{verbatim}
sage: G = C.digraph()
sage: G.show(edge_labels=true, figsize=12, vertex_size=1)
\end{verbatim}
or
\begin{verbatim}
sage: view(C, viewer = 'pdf', tightpage = True)
\end{verbatim}
The dual crystal $B(\Lambda_6)$ can be constructed as
\begin{verbatim}
sage: C = CrystalOfLetters(['E',6], dual = True)
\end{verbatim}
The crystal $B(\Lambda_7)$ of type $E_7$ can be accessed in a similar fashion.
Figure~\ref{fig:BLa7} was constructed as follows:
\begin{verbatim}
sage: C = CrystalOfLetters(['E',7])
sage: C.latex_file(filename.tex)
\end{verbatim}
\end{computerExample}

\begin{computerExample}
The classical crystals for type $E_6$ (and similarly for $E_7$) corresponding to arbitrary 
dominant weights can be constructed as follows:
\begin{verbatim}
sage: C = CartanType(['E',6])
sage: Lambda = C.root_system().weight_lattice().fundamental_weights()
sage: T = HighestWeightCrystal(C, dominant_weight=Lambda[1]+Lambda[6]+Lambda[2])
sage: T.highest_weight_vector()
[[1], [[2, -1], [1]], [6]]
sage: T.cardinality()          
34749
\end{verbatim}
\end{computerExample}

\begin{computerExample}
The Kirillov--Reshetikhin crystals $B^{r,s}$ for $r=1,6,2$ for type $E_6$ are also implemented:
\begin{verbatim}
sage: K = KirillovReshetikhinCrystal(['E',6,1], 1,1)
sage: K.cardinality()
27
sage: K = KirillovReshetikhinCrystal(['E',6,1], 6,1)

sage: K = KirillovReshetikhinCrystal(['E',6,1], 2,1)
sage: K.classical_decomposition()
Finite dimensional highest weight crystal of type ['E', 6] and 
dominant weight(s) [0, Lambda[2]]
sage: b = K.module_generator(); b
[[[2, -1], [1]]]
sage: b.e(0)
[]
sage: b.e(0).e(0)
[[[-1], [-2, 1]]]
\end{verbatim}
\end{computerExample}


\bigskip
\section{Outlook} \label{s:outlook}

In the case of $r = 3$, by~\cite{Chari:2001} the classical decomposition is $B^{r, s} \cong
\bigoplus_{\stackrel{j+k = s}{j,k \geq 0}} B(j\Lambda_3 + k \Lambda_6)$.  It is
possible to form a composition graph that includes nodes from both
$B(\Lambda_3)$ and $B(\Lambda_6)$ so that weakly increasing chains of vertices
correspond to $(I \setminus J)$-highest weight nodes.  However, it is
straightforward to verify that even for $s = 1$, the $I \setminus
\{1\}$-highest weight nodes are not uniquely determined by the statistics
$(\varepsilon_1, \ldots, \varepsilon_6, \varphi_1, \ldots, \varphi_6)$, in
contrast to the cases $r = 1, 2, 6$ that we have considered in this work.  
Hence one would first have to find vertices within each component which can
be distinguished using a suitable statistics, and then construct the corresponding
composition graph.
The case $r = 5$ is essentially the same as the $r = 3$ case.

The $\varepsilon$ and $\varphi$ statistics are the most obvious quantities
preserved by twisted isomorphism, and the fact that we were able to identify
highest weight nodes by their statistics allowed us to solve the equations that
proved our twisted isomorphism in fact had order three.

The classical decomposition of $B^{4,s}$ of type $E_6^{(1)}$ was conjectured 
in~\cite{HKOTY} and proven by Nakajima~\cite{nakajima:2003}.
As it involves more than two distinct fundamental weights, our tableau model and 
composition graphs would likely be substantially more complicated than those we have 
used for the cases $r = 1, 2, 6$.

As already mentioned in Section~\ref{ss:E7}, the method of composition graphs
for the adjoint Kirillov--Reshetikhin crystal $B^{1,s}$ of type $E_7^{(1)}$ is applicable and analogous
to type $E_6^{(1)}$. However, to prove that the result is indeed an affine combinatorial
crystal requires the analogue of Theorem~\ref{t:main} for twisted isomorphisms of
order two. The Dynkin diagram $E_8^{(1)}$ does not have nontrivial automorphisms.
Hence a new strategy is required.

It was conjectured in~\cite[Conjecture 2.1]{HKOTT} that the crystals $B^{r,s}$ of type $E_6^{(1)}$
are perfect. The proof for the crystals considered in this paper is still outstanding.

All Kirillov--Reshetikhin crystals can in principle be constructed from those of simply-laced
type using virtual crystals. In particular, the KR crystals for type $F_4^{(1)}$ and $E_6^{(2)}$ 
can be constructed from those of type $E_6^{(1)}$ (see~\cite[Example 3.1]{OSS:2003}).
Hence the construction of all type $E$ KR crystals is an important undertaking.


\bigskip
\appendix
\section{}\label{s:appendix}

Here, we set up and solve the linear equations describing how to map an $I
\setminus \{1\}$-highest weight node from component $j$ to an $\widetilde{I}
\setminus \{6,1\}$-highest weight node of component $k$, using the equation for
$\varphi_1 = j - (\wt_2 + \wt_3 + \wt_5)$ which must become $\varphi_6$ in the
image.  The cases correspond to the 6 maximal chains in the directed graph
$G_{6,1;0}$.

\bigskip

{\bf Case (1). }
{\tiny
\begin{eqnarray*}
a+b+c+d+e+f &=& k \\
\varphi_6 = a+b+2c+d &=& j - L_2 - L_3 -L_5 \\
L_2 &=& -a - b + f \\
L_3 &=& b \\
L_5 &=& -c+e \\
\end{eqnarray*}
with solution
\begin{eqnarray*}
f &=& (k-j)+L_2+L_3 \\
a &=& (k-j) \\
2c+d &=& 2j - k - 2 L_3 - L_2 - L_5 = \varphi_6 - (k-j) - L_3 \\
e &=& c + L_5 \\
\end{eqnarray*}
valid if $(k-j) + L_3 \leq \varphi_6$.
}

\medskip

{\bf Case (2). }
{\tiny 
\begin{eqnarray*}
a+b+c'+e+f &=& k \\
\varphi_6 = a+b &=& j - L_2 - L_3 -L_5 \\
L_2 &=& -a - b - c' + f \\
L_3 &=& b + c' \\
L_5 &=& c' + e \\
\end{eqnarray*}
with solution
\begin{eqnarray*}
f &=& (k-j) + L_2 + L_3 \\
a &=& (k-j) \\
b &=& 2j -k - L_2 - L_3 - L_5 = \varphi_6 - (k-j) \\
c' &=& -2j +k + 2 L_3 + L_2 + L_5 = L_3 + (k-j) - \varphi_6 \\
e &=& 2j -k - 2 L_3 - L_2 = L_5 - L_3 + \varphi_6 - (k-j) \\
\end{eqnarray*}
valid if $(k-j) \leq \varphi_6 \leq (k-j) + L_3 \leq \varphi_6 + L_5$.
}

\medskip

{\bf Case (3). }
{\tiny
\begin{eqnarray*}
a+b'+c'+e+f &=& k \\
\varphi_6 = a &=& j - L_2 - L_3 -L_5 \\
L_2 &=& -a - b' - c' + f \\
L_3 &=& c' \\
L_5 &=& b' + c' + e \\
\end{eqnarray*}
with solution
\begin{eqnarray*}
f &=& (k-j) + L_2 + L_3 \\
b' &=& -2j + k + L_2 + L_3 + L_5   = (k-j) - \varphi_6  \\
e  &=& 2j - k - L_2 - 2 L_3   = \varphi_6 - (k-j) - L_3 + L_5  \\
\end{eqnarray*}
valid if $\varphi_6 \leq (k-j) \leq (k-j) + L_3 \leq \varphi_6 + L_5$.
}

\medskip

{\bf Case (4). }
{\tiny 
\begin{eqnarray*}
a+b'+c'+e'+f &=& k \\
\varphi_6 = a &=& j - L_2 - L_3 -L_5 \\
L_2 &=& -a - b' - c' + f \\
L_3 &=& c' + e' \\
L_5 &=& b' + c' \\
\end{eqnarray*}
with solution
\begin{eqnarray*}
f &=& \varphi_6 + L_2 + L_5 = j - L_3 \\
e' &=& k - 2j + L_2 + 2 L_3 = (k-j) - \varphi_6 + L_3 - L_5 \\
c' &=& L_3 - e' = \varphi_6 - (k-j) + L_5 \\
b' &=& L_5 - c' = (k-j) - \varphi_6 \\
\end{eqnarray*}
valid if $\varphi_6 \leq (k-j) \leq \varphi_6 + L_5 \leq (k-j) + L_3$.
}

\medskip

{\bf Case (5). }
{\tiny

\begin{eqnarray*}
a+b+c'+e'+f &=& k \\
\varphi_6 = a+b &=& j - L_2 - L_3 -L_5 \\
L_2 &=& -a - b - c' + f \\
L_3 &=& b+c'+e' \\
L_5 &=& c' \\
\end{eqnarray*}
with solution
\begin{eqnarray*}
f &=& L_2 + L_5 + \varphi_6 = j - L_3 \\
e'&=& k - 2j + L_2 + 2 L_3 = (k-j) - \varphi_6 - L_5 + L_3\\
b &=& L_3 - L_5 - e' = 2j-k - L_2 - L_3 - L_5 = \varphi_6 - (k-j)  \\
a &=& \varphi_6 - b = (k-j) \\
\end{eqnarray*}
valid if $0 \leq (k-j) \leq \varphi_6 \leq \varphi_6 + L_5 \leq (k-j) + L_3$.
}

\medskip

{\bf Case (6). }
{\tiny

\begin{eqnarray*}
a+b'+c''+d'+e'+f &=& k \\
\varphi_6 = a &=& j - L_2 - L_3 -L_5 \\
L_2 &=& -a - b' + f \\
L_3 &=& -c''+e' \\
L_5 &=& b' \\
\end{eqnarray*}
with solution
\begin{eqnarray*}
f &=& \varphi_6 + L_2 + L_5 = j - L_3 \\
d' + 2 e' &=& (k-j) - \varphi_6 + 2 L_3 - L_5 \\
c'' &=& e' - L_3 \\
\end{eqnarray*}
valid if $\varphi_6 + L_5 < (k-j)$ because
\[ 0 \leq e'-L_3 \iff 0 \leq 2e'-2L_3 \iff 0 < 2e' - 2L_3 + d' \]
and $c'' \geq 0$ implies $e' \geq 0$.
}


\bigskip
\section{}\label{s:appendix_b}

Here, we prove that, whenever a set of parameters $(L_2, L_3, L_5, j, k)$ is
satisfied by two distinct cases from the systems described in
Appendix~\ref{s:appendix}, then the solutions we obtain in each case agree.

Since $\varphi_1 = k-j$ in every solution by Lemma~\ref{l:solve_equations} and
$\varphi_6 = j - (\wt_2+\wt_3+\wt_5)$ encodes $j$, we have that any solution $b
\in H_s^{\{6,1\};0}$ for the parameters $(L_2, L_3, L_5, j, k)$ must have prescribed
values for $(\wt_2(b), \wt_3(b), \wt_5(b), \varphi_1(b), \varphi_6(b))$.
Hence, to prove the uniqueness of the solution, it suffices to show that the
nodes of $H_s^{\{6,1\};0}$ are uniquely determined by $(\wt_2, \wt_3, \wt_5,
\varphi_1, \varphi_6)$.

\begin{proposition}
Let $b \in B(k \Lambda_2)$ and $b' \in B(k' \Lambda_2)$ be $\widetilde{I} \setminus
\{6,1\}$-highest weight nodes.  If $\wt_i(b) = \wt_i(b')$ for $i = 2, 3, 5$ and
$\varphi_j(b) = \varphi_j(b')$ for $j = 1, 6$, then $b = b'$.
\end{proposition}
\begin{proof}
Let $A = (a_{i,j})$ be the matrix where $a_{i,j}$ is the $i$th entry of $(\wt_2,
\wt_3, \wt_5, \varphi_1, \varphi_6)$ applied to the $j$th entry of $(a, b, b', c,
c', c'', d, d', e, e', f)$ from $G_{6,1;0}$.  Then,

\setcounter{MaxMatrixCols}{12}
\[
A = \begin{bmatrix}
-1 & -1 & -1 &  0 & -1 &  0 &  0 &  0 &  0 &  0 &  1 \\
 0 &  1 &  0 &  0 &  1 & -1 &  0 &  0 &  0 &  1 &  0 \\
 0 &  0 &  1 & -1 &  1 &  0 &  0 &  0 &  1 &  0 &  0 \\
 1 &  0 &  1 &  0 &  0 &  2 &  0 &  1 &  0 &  0 &  0 \\
 1 &  1 &  0 &  2 &  0 &  0 &  1 &  0 &  0 &  0 &  0 \\
\end{bmatrix}
\]

If there were two solutions for a given set of parameters $(\wt_2, \wt_3, \wt_5,
\varphi_1, \varphi_6)$ then we could subtract them to obtain a vector in the
nullspace of $A$.  Moreover, the positive coordinates of this vector would
correspond to nodes in $G_{6,1;0}$ that all lie on a maximal chain, and similarly
for the negative coordinates of the vector.

The nullspace of $A$ is spanned by the rows of the following
matrix.
\[
\begin{bmatrix}
a & b & b' & c & c' & c'' & d & d' & e & e' & f \\
1 & 0 & 0 & 0 & 0 & 0 & -1 & -1 & 0  & 0  & 1 \\
0 & 1 & 0 & 0 & 0 & 0 & -1 & 0  & 0  & -1 &  1 \\
0 & 0 & 1 & 0 & 0 & 0 & 0  & -1 & -1 & 0  & 1 \\
0 & 0 & 0 & 1 & 0 & 0 & -2 & 0  & 1  & 0  & 0 \\
0 & 0 & 0 & 0 & 1 & 0 & 0  & 0  & -1 & -1 & 1 \\
0 & 0 & 0 & 0 & 0 & 1 & 0  & -2 & 0  & 1  & 0 \\
\end{bmatrix}
\]
Although the nullspace is nontrivial, observe that no basis vector actually
corresponds to a valid relation because in every case we have that either the
positive entries or the negative entries in the basis vector violate the
constraint that the multiplicities lie on a maximal chain in $G_{6,1;0}$, so as
to form a weakly increasing tensor product.

Next, we show that these chain constraints are actually violated for every vector
in the nullspace of $A$.  To see this, consider that every minimal linear
dependence among the columns $\{u_1, \ldots, u_{11}\}$ of $A$ has the form
$\sum_{i=1}^{11} c_i u_i = 0$.  Define $\sgn(x)$ to be $0$, $-1$, or $1$, if $x$
is $0$, $<0$ or $>0$, respectively.  The collection of all sign vectors
$(\sgn(c_1), \ldots, \sgn(c_{11}))$ obtained from minimal linear dependencies
among the columns of $A$ forms what are known as the \em circuits \em of  an \em oriented
matroid\em.  Moreover, there is a formula to find these circuits that is given
in terms of certain minors of $A$.

To be precise, let $\{ v_1, \ldots, v_{11} \}$ denote the columns of $A$.
Then, we define $\chi_A : \{1, \cdots, 11\}^5 \rightarrow \{-1, 0, 1\}$ by
$\chi_A(i_1, \ldots, i_5) = \sgn \ \det(v_{i_1}, \ldots, v_{i_5})$.  Consider
\[ C : \{1, \ldots, 11\}^6 \rightarrow \{-1, 0, 1\}^{11} \]
where $C(i_1, \ldots, i_6)$ is defined by 
\[ ( \chi_A( (i_1, \ldots i_6) \setminus 1 ) (-1)^{j(1)+1}, \chi_A( (i_1, \ldots i_6) \setminus 2) (-1)^{j(2)+1}, \ldots, \chi_A( (i_1, \ldots i_6) \setminus 11) (-1)^{j(11)+1}). \]
Here, $j(m)$ denotes the index $j$ such that $i_j = m$, and we interpret
$\chi_A( (i_1, \ldots i_6) \setminus m )$ as 0 if $m \notin (i_1, \ldots,
i_6)$.  It then follows from \cite[Section 1.5]{om:1999} that the circuits are
precisely the set
\[ \{ C(i_1, \ldots, i_6) : (i_1, \ldots, i_6) \in \{1, \ldots, 11\}^6 \}
\setminus (0, 0, \ldots, 0). \]

Using this formula, we have computed that $A$ has 81 circuits and determined
that each of them violates the chain constraints from $G_{6,1;0}$.  Therefore, we
have that there is a unique solution for any given set of parameters $(\wt_2,
\wt_3, \wt_5, \varphi_1, \varphi_6)$.
\end{proof}


\bigskip

\begin{thebibliography}{BLVS{\etalchar{+}}99}

\bibitem[BFKL06]{BFKL:2006}
Georgia Benkart, Igor Frenkel, Seok-Jin Kang, and Hyeonmi Lee.
\newblock Level 1 perfect crystals and path realizations of basic
  representations at {$q=0$}.
\newblock {\em Internat. Math. Res. Not.}, pages Art. ID 10312, 28, 2006.

\bibitem[BLVS{\etalchar{+}}99]{om:1999}
Anders Bj{\"o}rner, Michel Las~Vergnas, Bernd Sturmfels, Neil White, and
  G{\"u}nter~M. Ziegler.
\newblock {\em Oriented matroids}, volume~46 of {\em Encyclopedia of
  Mathematics and its Applications}.
\newblock Cambridge University Press, Cambridge, second edition, 1999.

\bibitem[Cha01]{Chari:2001}
Vyjayanthi Chari.
\newblock On the fermionic formula and the {K}irillov-{R}eshetikhin conjecture.
\newblock {\em Internat. Math. Res. Notices}, (12):629--654, 2001.

\bibitem[FOS08]{fos:2009b}
Ghislain Fourier, Masato Okado, and Anne Schilling.
\newblock Perfectness of {K}irillov--{R}eshetikhin crystals for nonexceptional
  types.
\newblock {\em Contemp. Math., to appear}, 2008.

\bibitem[FOS09]{fos:2009a}
Ghislain Fourier, Masato Okado, and Anne Schilling.
\newblock {K}irillov-{R}eshetikhin crystals for nonexceptional types.
\newblock {\em Advances in Mathematics}, 222:1080--1116, 2009.

\bibitem[Gre07]{green:2007}
R.~M. Green.
\newblock Full heaps and representations of affine {K}ac-{M}oody algebras.
\newblock {\em Int. Electron. J. Algebra}, 2:137--188 (electronic), 2007.

\bibitem[Gre08]{green:2008}
R.~M. Green.
\newblock Full heaps and representations of affine {W}eyl groups.
\newblock {\em Int. Electron. J. Algebra}, 3:1--42, 2008.

\bibitem[HK02]{hk:2002}
Jin Hong and Seok-Jin Kang.
\newblock {\em Introduction to quantum groups and crystal bases}, volume~42 of
  {\em Graduate Studies in Mathematics}.
\newblock American Mathematical Society, Providence, RI, 2002.

\bibitem[HKO{\etalchar{+}}99]{HKOTY}
G.~Hatayama, A.~Kuniba, M.~Okado, T.~Takagi, and Y.~Yamada.
\newblock Remarks on fermionic formula.
\newblock In {\em Recent developments in quantum affine algebras and related
  topics ({R}aleigh, {NC}, 1998)}, volume 248 of {\em Contemp. Math.}, pages
  243--291. Amer. Math. Soc., Providence, RI, 1999.

\bibitem[HKO{\etalchar{+}}02]{HKOTT}
Goro Hatayama, Atsuo Kuniba, Masato Okado, Taichiro Takagi, and Zengo Tsuboi.
\newblock Paths, crystals and fermionic formulae.
\newblock In {\em Math{P}hys odyssey, 2001}, volume~23 of {\em Prog. Math.
  Phys.}, pages 205--272. Birkh\"auser Boston, Boston, MA, 2002.

\bibitem[HN06]{hernandez_nakajima:2006}
David Hernandez and Hiraku Nakajima.
\newblock Level 0 monomial crystals.
\newblock {\em Nagoya Math. J.}, 184:85--153, 2006.

\bibitem[Hos07]{hoshino:2007}
Ayumu Hoshino.
\newblock Generalized {L}ittlewood-{R}ichardson rule for exceptional {L}ie
  algebras {$E_6$} and {$F_4$}.
\newblock In {\em Lie algebras, vertex operator algebras and their
  applications}, volume 442 of {\em Contemp. Math.}, pages 159--169. Amer.
  Math. Soc., Providence, RI, 2007.

\bibitem[Kas95]{K:1995}
Masaki Kashiwara.
\newblock On crystal bases.
\newblock In {\em Representations of groups ({B}anff, {AB}, 1994)}, volume~16
  of {\em CMS Conf. Proc.}, pages 155--197. Amer. Math. Soc., Providence, RI,
  1995.

\bibitem[Kas02]{K:2002}
Masaki Kashiwara.
\newblock On level-zero representations of quantized affine algebras.
\newblock {\em Duke Math. J.}, 112(1):117--175, 2002.

\bibitem[Kas05]{K:2005}
Masaki Kashiwara.
\newblock Level zero fundamental representations over quantized affine algebras
  and {D}emazure modules.
\newblock {\em Publ. Res. Inst. Math. Sci.}, 41(1):223--250, 2005.

\bibitem[KKM{\etalchar{+}}92]{KKMMNN:1992}
Seok-Jin Kang, Masaki Kashiwara, Kailash~C. Misra, Tetsuji Miwa, Toshiki
  Nakashima, and Atsushi Nakayashiki.
\newblock Perfect crystals of quantum affine {L}ie algebras.
\newblock {\em Duke Math. J.}, 68(3):499--607, 1992.

\bibitem[KMOY07]{KMOY:2007}
M.~Kashiwara, K.~C. Misra, M.~Okado, and D.~Yamada.
\newblock Perfect crystals for {$U\sb q(D\sp {(3)}\sb 4)$}.
\newblock {\em J. Algebra}, 317(1):392--423, 2007.

\bibitem[KN94]{KN}
Masaki Kashiwara and Toshiki Nakashima.
\newblock Crystal graphs for representations of the {$q$}-analogue of classical
  {L}ie algebras.
\newblock {\em J. Algebra}, 165(2):295--345, 1994.

\bibitem[Kod08]{kodera}
Ryosuke Kodera.
\newblock A generalization of adjoint crystals for the quantized affine
  algebras of type {$A_{n}^{(1)}$}, {$C_{n}^{(1)}$} and {$D_{n+1}^{(2)}$}.
\newblock {\em preprint {\tt arXiv:0802.3964}}, 2008.

\bibitem[Lit96]{littelmann:1996}
Peter Littelmann.
\newblock A plactic algebra for semisimple {L}ie algebras.
\newblock {\em Adv. Math.}, 124(2):312--331, 1996.

\bibitem[LP08]{lenart-postnikov}
Cristian Lenart and Alexander Postnikov.
\newblock A combinatorial model for crystals of {K}ac-{M}oody algebras.
\newblock {\em Trans. Amer. Math. Soc.}, 360(8):4349--4381, 2008.

\bibitem[LS86]{LS:1986}
V.~Lakshmibai and C.~S. Seshadri.
\newblock Geometry of {$G/P$}. {V}.
\newblock {\em J. Algebra}, 100(2):462--557, 1986.

\bibitem[Mag06]{magyar:2006}
Peter Magyar.
\newblock Littelmann paths for the basic representation of an affine {L}ie
  algebra.
\newblock {\em J. Algebra}, 305(2):1037--1054, 2006.

\bibitem[Nak03]{nakajima:2003}
Hiraku Nakajima.
\newblock {$t$}-analogs of {$q$}-characters of {K}irillov-{R}eshetikhin modules
  of quantum affine algebras.
\newblock {\em Represent. Theory}, 7:259--274 (electronic), 2003.

\bibitem[OS08]{OS:2008}
Masato Okado and Anne Schilling.
\newblock Existence of {K}irillov-{R}eshetikhin crystals for nonexceptional
  types.
\newblock {\em Represent. Theory}, 12:186--207, 2008.

\bibitem[OSS03]{OSS:2003}
Masato Okado, Anne Schilling, and Mark Shimozono.
\newblock Virtual crystals and {K}leber's algorithm.
\newblock {\em Comm. Math. Phys.}, 238(1-2):187--209, 2003.

\bibitem[SCc09]{sage-combinat}
The {S}age-{C}ombinat community.
\newblock {S}age-{C}ombinat: enhancing {S}age as a toolbox for computer
  exploration in algebraic combinatorics, 2009.
\newblock {\url{http://combinat.sagemath.org}}.

\bibitem[Sch08]{schilling:2008}
Anne Schilling.
\newblock Combinatorial structure of {K}irillov-{R}eshetikhin crystals of type
  {$D\sp {(1)}\sb n,B\sp {(1)}\sb n,A\sp {(2)}\sb {2n-1}$}.
\newblock {\em J. Algebra}, 319(7):2938--2962, 2008.

\bibitem[Ste03]{stembridge:2003}
John~R. Stembridge.
\newblock A local characterization of simply-laced crystals.
\newblock {\em Trans. Amer. Math. Soc.}, 355(12):4807--4823 (electronic), 2003.

\bibitem[WSea09]{sage}
{W}illiam {S}tein~et al.
\newblock {S}age {M}athematics {S}oftware (version 4.1.1), 2009.
\newblock {\url{http://www.sagemath.org}}.

\end{thebibliography}

\newcommand{\etalchar}[1]{$^{#1}$}

\end{document}